\documentclass[12pt]{amsart}
\usepackage{url}

\usepackage{verbatim}
\usepackage{amssymb}
\usepackage{upref}
\usepackage[all]{xy}
\usepackage{color}

\allowdisplaybreaks

\newtheorem{thmsub}{Theorem}[subsection]
\newtheorem{lemsub}[thmsub]{Lemma}
\newtheorem{corsub}[thmsub]{Corollary}
\newtheorem{propsub}[thmsub]{Proposition}
\newtheorem{defnsub}[thmsub]{Definition}
\newtheorem{notnsub}[thmsub]{Notation}
\newtheorem{remsub}[thmsub]{Remark}
\newtheorem{remssub}[thmsub]{Remarks}

\newtheorem{thm}{Theorem}[section]
\newtheorem*{thm*}{Theorem}
\newtheorem{lem}[thm]{Lemma}
\newtheorem{cor}[thm]{Corollary}
\newtheorem{prop}[thm]{Proposition}

\theoremstyle{definition}
\newtheorem{defn}[thm]{Definition}
\newtheorem{q}[thm]{Question}

\newtheorem{notn}[thm]{Notation}
\newtheorem{hyp}[thm]{Hypothesis}
\newtheorem*{notn*}{Notation}

\newtheorem*{hyp*}{Hypothesis}

\newtheorem{rem}[thm]{Remark}
\newtheorem*{rem*}{Remark}

\numberwithin{equation}{section}

\newcommand{\corref}[1]{Corollary~\textup{\ref{#1}}}
\newcommand{\lemref}[1]{Lemma~\textup{\ref{#1}}}
\newcommand{\propref}[1]{Proposition~\textup{\ref{#1}}}

\newcommand{\defnref}[1]{Definition~\textup{\ref{#1}}}
\newcommand{\remref}[1]{Remark~\textup{\ref{#1}}}

\newcommand{\hypref}[1]{Hypothesis~\textup{\ref{#1}}}

\newcommand{\midtext}[1]{\quad\text{#1}\quad}

\renewcommand{\and}{\midtext{and}}

\renewcommand{\)}{\textup)}
\newcommand{\ie}{\emph{i.e.}}

\newcommand{\cf}{\emph{cf.}}

\newcommand{\N}{\mathbb N}

\newcommand{\C}{\mathbb C}

\newcommand{\CC}{\mathcal C}
\newcommand{\DD}{\mathcal D}
\newcommand{\KK}{\mathcal K}

\newcommand{\EE}{\mathcal E}

\newcommand{\BB}{\mathcal B}

\newcommand{\RR}{\mathcal R}

\newcommand{\MM}{\mathcal M}
\newcommand{\NN}{\mathcal N}
\newcommand{\UU}{\mathcal U}

\renewcommand{\AA}{\mathcal A}

\renewcommand{\t}{\theta}

\renewcommand{\epsilon}{\varepsilon}

\newcommand{\D}{\Delta}

\DeclareMathOperator{\ad}{Ad}

\DeclareMathOperator{\obj}{Obj}

\DeclareMathOperator{\maxi}{Max}

\DeclareMathOperator{\nor}{Nor}

\DeclareMathOperator{\inc}{Inc}

\newcommand{\id}{\text{\textup{id}}}

\newcommand{\inv}{^{-1}}

\renewcommand{\bar}{\overline}
\newcommand{\what}{\widehat}

\newcommand{\cs}%
{\ensuremath{\mathbf{C^*}}}
\newcommand{\csnd}%
{\ensuremath{\cs\!\!_\mathbf{nd}}}
\newcommand{\coact}%
{\ensuremath{\mathbf{C^*coact}}}
\newcommand{\coactnd}%
{\ensuremath{\coact_\mathbf{nd}}}
\newcommand{\coactn}%
{\ensuremath{\coact^\mathbf{n}}}
\newcommand{\coactnnd}%
{\ensuremath{\coactn_\mathbf{nd}}}
\newcommand{\coactm}%
{\ensuremath{\coact^\mathbf{m}}}
\newcommand{\coactmnd}%
{\ensuremath{\coactm_\mathbf{nd}}}

\newcommand{\dn}{\downarrow}

\newcommand{\CP}{\textup{CP}}

\newcommand{\cp}{\CP}

\renewcommand{\coact}{\ensuremath{\CC(G)}}
\newcommand{\ncoact}{\ensuremath{\CC^n(G)}}
\newcommand{\mcoact}{\ensuremath{\CC^m(G)}}

\newtheorem{ppies}[thm]{Properties}
\def\AD{(A,\Delta)} 
\def\ADu{(\au,\Delta_{\rm u})}
\def\hr{{h_{{\rm r}}}} 

\def\au{{A_{{\rm u}}}}

\def\nuu#1{\|{#1\|}_{{\rm u}}}

\begin{document}
\title{Reflective-coreflective equivalence}
\author{Erik B\'edos}
\address{Institute of Mathematics, University of Oslo, P.B. 1053 Blindern, 0316 Oslo, Norway}
\email{bedos@math.uio.no}
\author{S. Kaliszewski}
\address{School of Mathematical and Statistical Sciences, Arizona State University, Tempe, AZ 85287}
\email{kaliszewski@asu.edu}
\author{John Quigg}
\address{School of Mathematical and Statistical Sciences, Arizona State University, Tempe, AZ 85287}
\email{quigg@asu.edu}
\date{Revised version, February~18, 2011}



\begin{abstract}
We explore a curious type of equivalence between certain pairs of reflective and coreflective subcategories. We illustrate with examples involving noncommutative duality for $C^*$-dynamical systems and compact quantum groups, as well as examples where the subcategories are actually isomorphic.
\end{abstract}

\maketitle

\section{Introduction}
\label{intro}

Our intent in writing this paper is to explore a special 
type of equivalence between certain pairs of reflective and coreflective subcategories.
We have noticed that in certain categories involving $C^*$-algebras, there is a pair of equivalent subcategories, one reflective and the other coreflective, and moreover this equivalence really depends only upon 
certain categorical properties, and not upon the theory of $C^*$-algebras.
To highlight the categorical nature of this phenomenon, we will present the equivalence from a purely abstract category-theoretical point of view, and then describe several examples from $C^*$-algebra theory.

%We are operator algebraists, and can only claim an amateur-level expertise in category theory.
%It seems to us entirely possible that the type of equivalence we present here is known to category theorists, but we have been frustrated by our inability to find it in the literature. We would be interested to learn of other instances of this equivalence.

To give an idea of what our equivalence entails, consider subcategories $\MM$ and $\NN$ of a category $\CC$. Letting  $\inc_\MM$ denote the inclusion functor
of $\MM$ into $\CC$ and  $F|_\MM = F \circ \inc_\MM: \MM \to \DD$ the restriction  of a functor $F:\CC\to\DD$, let us say that $\MM$ and $\NN$ are $\CC$\emph{-equivalent} if $\MM=\NN$ or  there exist functors $S:\CC\to\MM$, $T:\CC\to\NN$ such that $T|_\MM:\MM \to \NN$ and $S|_\NN:\NN\to\MM$ are quasi-inverses of each other. Clearly, $\MM$ and $\NN$ are then equivalent as categories in the usual sense and it is easy to check that $\CC$-equivalence is an equivalence relation. Now assume that $\MM$ is  coreflective  in $\CC$ and  $\NN$ is  reflective in $\CC$, with coreflector
$M:\CC\to\MM$ and reflector $N:\CC\to\NN$, respectively. 
The restriction  $N|_\MM$ is then a left adjoint of $M|_\NN$, and in order to prove that  $\MM$ and $\NN$ are $\CC$-equivalent, it suffices to show that  the adjunction $N|_\MM\dashv M|_\NN$ is an adjoint equivalence. When this happens may be characterized  in several  ways.   
One of them involves
the counit $\psi$ of the adjunction $\inc_\MM\dashv M$
and the unit $\theta$ of the adjunction $N\dashv \inc_\NN$,
which enjoy certain universal properties by definition. We show in Section \ref{equivalence} that $N|_\MM\dashv M|_\NN$  is an adjoint equivalence if and only if, when everything is restricted to the subcategories, each of $\psi$ and $\theta$ actually possesses both universal properties.

As kindly indicated to us by Ross Street after reading an earlier draft of this article, our considerations may be enlarged to characterize when a composite adjunction gives an adjoint equivalence. For the benefit of specialists in category theory, we have devoted a separate section
(Section \ref{Ross}) to this more general approach. 
We thank Professor Street warmly for his permission to include this material. 
 
Our first  main example of the reflective-coreflective equivalence involves normal and maximal coactions of a locally compact group on $C^*$-algebras. 
It has already appeared in the literature \cite{clda}, but we provide an alternative development, with several improvements arising from a close scrutiny of the underlying category theory. 
To avoid interrupting the exposition of this equivalence, 
we have relegated the prerequisite background on  coactions and their crossed products to an appendix.

Our second example deals with reduced and universal compact quantum groups. The equivalence of the two associated categories is surely known to experts in quantum group theory, but does not seem to be mentioned in the existing literature.  We also include two other examples involving tensor products of $C^*$-algebras and group representations, in which the subcategories are not only equivalent but in fact isomorphic.

\vspace{-2ex}
\section{Preliminaries}
\label{prelim}

We record here our conventions regarding category theory.
All of this can be found in \cite{maclane}.
We assume familiarity with elementary category theory, e.g., adjoint functors, coreflective and reflective subcategories.
However, since we want this paper to be readable by operator algebraists, among others, we give somewhat more detail in this preliminary section than might seem customary to a category theorist.

\begin{notn}
If $\CC$ and $\DD$ are categories, 
we write:
\begin{enumerate}
\item $\obj\CC$ for the class of objects in $\CC$;

\item $\CC(x,y)$ for the set of morphisms with domain 
$x\in\obj\CC$ and codomain $y\in\obj\CC$, 
and $f : x \to y$ in $\CC$ to mean $f \in \CC(x,y)$;

\item $1_x$ for the identity morphism of the object $x$;

\item (most of the time) $Ff$ rather than $F(f)$ for the value of a functor $F:\CC\to\DD$ at a morphism $f$ (although we usually write compositions of morphisms as $f\circ g$ rather than 
%the categorists' 
$fg$).
\end{enumerate}
\end{notn}

Recall that a functor $F:\CC\to\DD$ is called
\emph{full} (respectively, \emph{faithful}) if it maps $\CC(x,y)$ surjectively (respectively, injectively) to $\DD(Fx, Fy)$ for all $x, y \in \obj\CC$, and  
\emph{essentially surjective} if every object in $\DD$ is isomorphic to one in the image of $F$.

\medskip 
If $x\in\obj\CC$ and $G:\DD\to\CC$ is a functor, we write $x\dn G$ for the \emph{comma category} 
whose objects are pairs $(y,f)$, where $y\in \obj\DD$ and $f:x\to Gy$ in $\CC$,
and in which $h:(y,f)\to (z,g)$ means that $h:y\to z$ in $\DD$ and $(Gh)\circ f=g$.
Dually, we write $G\dn x$ for the comma category
whose objects are pairs $(y,f)$, where $y\in \obj\DD$ and $f:Gy\to x$ in $\CC$,
and in which $h:(y,f)\to (z,g)$ means that $h:y\to z$ in $\DD$ and $g\circ(Gh)=f$.
If $\inc_\DD:\DD\hookrightarrow\CC$ is an inclusion functor, we write
\begin{align*}
x\dn\DD=x\dn \inc_\DD
\midtext{and}
\DD\dn x=\inc_\DD\dn x.
\end{align*}
In the particular case that $\DD=\CC$, 
the  categories  $x\dn\DD$ and $\DD\dn x$ are sometimes called \emph{slice categories}.

A more general definition goes follows: given functors $F:\CC \to \DD$ and $G:\EE \to \DD$, the \emph{comma category} $F\dn G$ has as objects
all triples $(x,y,f)$, with  $x\in \obj\CC$, $y\in \obj\EE$ and $f:Fx\to Gy$ in $\DD$, and as morphisms $(x,y,f)\to (x',y',f')$ all pairs $(k,h)$ of morphisms $k:x\to x'$ in $\CC$, $h:y \to y'$ in $\EE$ such that $f' \circ (Fk) = (Gh) \circ f$. The composite $(k',h')\circ (k,h)$ is $(k'\circ k, h'\circ h)$, when defined.

\medskip Recall that if $x\in\obj\CC$ and $G:\DD\to\CC$ is a functor,
a \emph{universal morphism from $x$ to $G$} is an initial object in the comma category $x\dn G$, and, dually, a \emph{universal morphism from $G$ to $x$} is a final object in $G\dn x$.
%If $G=\inc_\DD:\DD\hookrightarrow\CC$, we refer to universal morphisms from $x$ to $\DD$, or from $\DD$ to $x$. A morphism
%$f:x\to y$ in $\CC$ is an isomorphism if and only if $(y,f)$ is a universal morphism from $x$ to $\CC$. The dual is also true. 

Thus, a universal morphism $(u,\eta)$ from $x$ to $G$ is characterized by the following universal property: whenever $f:x\to Gy$ in $\CC$
there is a unique morphism $g$ in $\DD$ making the diagram
\[
\xymatrix{
x \ar[r]^-\eta \ar[dr]_f
&Gu \ar@{-->}[d]^{Gg}
&u \ar@{-->}[d]^g_{!}
\\
&Gy
&y
}
\]
commute,
and, dually, a universal morphism $(u,\epsilon)$ from $G$ to $x$ is characterized by the universal property that whenever $f:Gy\to x$ in $\CC$ there is a unique morphism $g$ in $\DD$ making the diagram
\[
\xymatrix{
Gy \ar@{-->}[d]_{Gg} \ar[dr]^f
&&y \ar@{-->}[d]^g_{!}
\\
Gu \ar[r]_-\epsilon
&x
&u
}
\]
commute.

Also,
$(u,\eta)$ is a universal morphism from $x$ to $G$ if and only if for every $y\in\obj\DD$ the map $\phi:\DD(u,y)\to\CC(x,Gy)$ defined by
\[
\phi(g)=(Gg)\circ\eta
\]
is bijective,
in which case we have $\eta=\phi(1_u)$.
%\[
%\eta=\phi(1_u),
%\]
Dually, $(u,\epsilon)$ is a universal morphism  from $G$ to $x$ if and only if for every $y\in\obj\DD$ the map $\psi:\DD(y,u)\to\CC(Gy,x)$ defined by
\[
\psi(g)=\epsilon\circ Gg
\]
is bijective,
in which case we have $\, \epsilon=\psi(1_u).$
%\[
%\epsilon=\psi(1_u).
%\]
%\smallskip 
If $G=\inc_\DD:\DD\hookrightarrow\CC$, we refer to universal morphisms from $x$ to $\DD$, or from $\DD$ to $x$. A morphism
$\eta:x\to u$ in $\CC$ is an isomorphism if and only if $(u,\eta)$ is a universal morphism from $x$ to $\CC$. The dual statement is also true. 

\smallskip A functor $F:\CC\to\DD$ is a \emph{left adjoint} of 
a functor $G:\DD\to\CC$,
or $G$ is a \emph{right adjoint} of $F$,
if there are bijections
\[
\phi_{x,y}:\DD(Fx,y)\to\CC(x,Gy)
\midtext{for all $x\in\obj\CC,y\in\obj\DD$}
\]
that are natural in $x$ and $y$.
In this case, we write `$F\dashv G$', and refer to $F\dashv G$ as an \emph{adjunction}
from $\CC$ to $\DD$.  
As is customary, we usually drop the subscripts $x,y$ from the $\phi$, which causes no confusion.

If $F\dashv G$, 
with natural bijections $\phi:\DD(Fx,y)\to\CC(x,Gy)$,
then for every $x\in\obj\CC$ the pair $(Fx,\eta_x)$ is a universal morphism from $x$ to $G$, where $\eta_x=\phi(1_{Fx})$;
and for every $y\in\obj\DD$ the pair $(Gy,\epsilon_y)$ is a universal morphism from $F$ to $y$, where $\epsilon_y=\phi\inv(1_{Gy})$.
Recall that
$\eta:1_\CC\to GF$ is called the \emph{unit} of the adjunction $F\dashv G$,
and $\epsilon:FG\to 1_\DD$ is the \emph{counit}.

Conversely,
given a functor $G:\DD\to\CC$, if for each $x\in\obj\CC$ we have a universal morphism $(Fx,\eta_x)$ from $x$ to $G$,
then the map $F$ on objects extends uniquely to a functor
such that $\eta:1_\CC\to GF$ is a natural transformation,
and moreover $F\dashv G$,
with natural bijections $\phi:\DD(Fx,y)\to\CC(x,Gy)$ defined by
$\phi(g)=Gg\circ\eta_x$.\footnote{Dually, if $F:\CC\to\DD$ is a functor, and if for all $y\in\obj\DD$ we have a universal morphism $(Gy,\epsilon_y)$ from $F$ to $y$, then the map $G$ extends uniquely to a functor
such that $\epsilon:FG\to 1_\DD$ is a natural transformation,
and moreover $F\dashv G$,
with natural bijections $\phi:\DD(Fx,y)\to\CC(x,Gy)$ determined by
$\phi\inv(f)=\epsilon_y\circ Ff$.}

If, given $G\colon\DD\to\CC$, 
we only know that for every $x\in\obj\CC$ there \emph{exists} a universal morphism from~$x$ to~$G$, 
then an Axiom of Choice for classes says that we can choose one such universal morphism $(Fx,\eta_x)$ for every $x$;
thus $G$ is left-adjointable 
if and only if every $x\in\obj\CC$ has a universal morphism to $G$.
Dually, a given functor $F\colon\CC\to\DD$ is right-adjointable 
if and only if every $y\in\obj\DD$ has a universal morphism from~$F$.

It follows that $F\dashv G$ if and only if there exists a natural transformation $\eta:1_\CC\to GF$ such that,
for every $x\in\obj\CC$, the pair $(Fx,\eta_x)$ is a universal morphism from $x$ to $G$. 
There is a similar characterization in terms of $\epsilon$.

For any functor $G:\DD\to\CC$, the left adjoints of $G$ form a natural isomorphism class, and dually for any functor $F:\CC\to\DD$, the right adjoints of $F$ form a natural isomorphism class.

\medskip
Certain properties of adjoints are related to properties of the unit and counit, as illustrated in the following standard lemma.

\begin{lem}
\label{properties}
Let $F\dashv G$, with unit $\eta$ and counit $\epsilon$.
\begin{enumerate}
\item $F$ is faithful if and only if every $\eta_x$ is a monomorphism.

\item $F$ is full if and only if every $\eta_x$ is a split epimorphism.

\item $F$ is full and faithful if and only if $\eta:1_\CC\to GF$ is a natural isomorphism.

\item $G$ is faithful if and only if every $\epsilon_y$ is an epimorphism.

\item $G$ is full if and only if every $\epsilon_y$ is a split monomorphism.

\item $G$ is full and faithful if and only if $\epsilon:FG\to 1_\DD$ is a natural isomorphism.
\end{enumerate}
\end{lem}
Adjunctions can be composed: if $F:\CC\to\DD$, $G:\DD\to\CC$,$H:\DD\to\EE$, and $K:\EE\to\DD$ are functors with $F\dashv G$ and $H\dashv K$, then $H\circ F\dashv G\circ K$.

Recall that if $F:\CC\to\DD$ is an \emph{equivalence}, so that there is a functor $G:\DD\to\CC$ such that%
\footnote{Between functors, the symbol `$\cong$' always denotes natural isomorphism.}
$GF\cong 1_\CC$ and $FG\cong 1_\DD$, then $F$ and $G$ are called \emph{quasi-inverses} of each other; 
$F$ and $G$ are then left and right adjoint of each other, and $\CC$ and $\DD$ are called \emph{equivalent}.

An adjunction $F\dashv G$  from $\CC$ to $\DD$ is called an \emph{adjoint equivalence} if both its unit and counit are natural isomorphisms, \ie, if both $F$ and~$G$ are  full and faithful (using Lemma \ref{properties}); clearly, $F$ and $G$ are then quasi-inverses, and  $\CC$ and $\DD$ are equivalent.   

A functor $F:\CC\to\DD$ is an equivalence if only if it is full, faithful, and essentially surjective, in which case a functor from $\DD$ to $\CC$ is a quasi-inverse of $F$ if and only if it is a right adjoint of $F$, if and only if it is a left adjoint of $F$.

\subsection*{Subcategories}
\label{subcategories}

A subcategory $\NN$ of $\CC$ is \emph{reflective} if the inclusion functor
$\inc_\NN:\NN\to\CC$
is left-adjointable, and any left adjoint $N$
of $\inc_\NN$ is then called
a \emph{reflector} of $\CC$ in $\NN$.

Such a reflector $N:\CC\to\NN$ is completely determined by the choice of a universal morphism $(Nx,\theta_x)$ from  $x$ to  $\NN$ for each object $x$ of $\CC$. 
The universal property says that every morphism in $\CC$ from $x$ to an object in $\NN$ factors uniquely through $\theta_x$:
\[
\xymatrix{
x \ar[r]^{\theta_x} \ar[dr]
&Nx \ar@{-->}[d]_{!}
\\
&z\in\NN.
}
\] 
Hence, if $f:x\to y$ in $\CC$, then $Nf$ is the unique morphism in $\NN$ making the diagram
\[
\xymatrix{
x \ar[r]^-{\theta_x} \ar[d]_f
&Nx \ar@{-->}[d]^{Nf}_{!}
\\
y \ar[r]_-{\theta_y}
&Ny
}
\]
commute.

The associated natural transformation $\theta : 1_\CC \to \inc_\NN \circ N$ is then the unit of the adjunction $N\dashv \inc_\NN$. Its counit $\rho : N \circ \inc_\NN \to 1_\NN$ is the natural transformation  given by letting  $\rho_y:Ny\to y$ be the unique morphism in $\NN$ such that $\rho_y \circ \theta_y = 1_y$ for each $y\in \obj \NN$, and $(y, \rho_y)$ is then a universal morphism from $N$ to $y$. 

Note that if $\NN$ is full, then $\inc_\NN$ is full and faithful, so the counit $\rho$ is a natural isomorphism and  $\theta_y = \rho_y^{-1}$ when $y \in \obj \NN$.
In this case we could in fact  choose $\theta_y=1_y$, \ie, we could arrange that $N|_\NN=1_\NN$; the counit $\rho$ would then be just the identity transformation and the reflector $N:\CC\to\NN$ could be thought of as a sort of ``projection'' of $\CC$ onto  $\NN$.

Dually,
a subcategory $\MM$ of $\CC$ is \emph{coreflective} if the inclusion functor $\inc_\MM:\MM\to\CC$ is right-adjointable, and any right adjoint $M$
of $\inc_\MM$ is called
a \emph{coreflector} of $\CC$ in $\MM$.

Such a coreflector $M$ is completely determined by the choice of a universal morphism $(Mx ,\psi_x)$ from $\MM$ to $x$ for each object $x$ of $\CC$. The universal property says that every morphism from an object of $\MM$ to $x$ factors uniquely through $\psi_x$:
\[
\xymatrix{
y\in\MM \ar[dr] \ar@{-->}[d]^{!}
\\
Mx \ar[r]_{\psi_x}
&x.
}
\]
Hence, if $g:z\to x$ in $\CC$, then $Mg$ is the unique morphism in $\MM$ making the diagram
\[
\xymatrix{
Mz \ar[r]^-{\psi_z} \ar@{-->}[d]_{Mg}^{!}
&z \ar[d]^g
\\
Mx \ar[r]_-{\psi_x}
&x
}
\]
commute.

The associated natural transformation $\psi :   \inc_\MM \circ M \to 1_\CC$ is then the counit of the adjunction $\inc_\MM\dashv M$. Its unit $\sigma :  1_\MM
\to M \circ \inc_\MM$ is the natural transformation  given by letting  $\sigma_y: y \to My$ be the unique morphism in $\MM$ such that $\psi_y \circ \sigma_y = 1_{y}$ 
for each $y\in \obj \MM$, and $(y, \sigma_y)$ is then a universal morphism from $y$ to $M$. 

Similarly to reflective subcategories, note that if $\MM$ is full, then the unit $\sigma$ is a natural isomorphism and  $\psi_y = \sigma_y^{-1}$ when $y \in \obj \MM$.
In this case we could  choose $\psi_y=1_y$, \ie, we could arrange that $M|_\MM=1_\MM$ and  the coreflector $M:\CC\to\MM$ could be thought of as a sort of ``projection'' of $\CC$ onto  $\NN$. However, since this projection property can also be made to happen with reflective subcategories, it is not terrifically informative.

\section{Composite adjoint equivalence}
\label{Ross}

Throughout this section we consider categories $\CC$, $\MM$, and $\NN$,
and functors $I$, $J$, $M$, and $N$ as shown:
%and functors $$N:\CC\to\NN, \, J :\NN \to \CC, \, M:\CC\to\MM,\,  I:\MM\to\CC$$
\[
\xymatrix@C+30pt{
\MM \ar@<1ex>[r]^{I} \ar@/^2pc/[rr]^{NI}
&\CC \ar@<1ex>[r]^N \ar@<1ex>[l]^M
&\NN \ar@<1ex>[l]^{J} \ar@/^2pc/[ll]^{MJ}
}
\, { }
\]
We further assume that $N\dashv J$ and $I \dashv M$;
this provides us with units  $\theta:1_\CC\to J N$ and $\sigma :  1_\MM\to M I$,
and counits $\rho : N J \to 1_\NN$ and  $\psi:I M\to 1_\CC$. 

The unit $\eta:1_\MM \to MJ NI$ and the counit $\epsilon: NI MJ \to 1_\NN$ for the composite adjunction $NI\dashv MJ$ are given by  
$$\eta = M\theta I \cdot \sigma\, , \, \, \epsilon = \rho \cdot N\psi J \,.$$ In other words, we have:
\begin{align}\label{eq-u} 
\eta_x&= (M\theta_{Ix}) \circ \sigma_x \ \ \text{for each}\ x \in \obj \MM\,;\\ \label{eq-cu}
\epsilon_y&= \rho_y \circ (N\psi_{Jy})  \  \ \text{for each}\ y  \in \obj \NN.
 \end{align} 

We are interested in conditions ensuring that $NI\dashv MJ$ is an adjoint equivalence,
{\it i.e.}, in $\eta$ and $\epsilon$ being natural isomorphisms.

\begin{thm}
\label{main-c}
The following conditions are equivalent:

\begin{itemize}
\item[(i)] 
For each $x\in\obj\MM$, $(x,\theta_{Ix})$ is a final object in $I\dn JNIx$.
\item[(ii)] $NI$ is full and faithful.
\item[(iii)] $\eta$ is a natural isomorphism.
\end{itemize}

If $I$   is full and faithful, then conditions {\rm (i) -- (iii)} are equivalent to
\begin{itemize}
\item[(iv)] $ M\theta_{Ix}$ is an isomorphism in $\MM$ for each $ x \in \obj \MM$.
\end{itemize} 

Similarly, the following conditions are equivalent:

\begin{itemize}
\item[(v)] For each $y\in\obj\NN$, 
$(y,\psi_{Jy})$ is an initial object in $IMJy\dn J$.
\item[(vi)] $MJ$ is full and faithful.
\item[(vii)] $\epsilon$ is a natural isomorphism. 
\end{itemize}

If $J$ is full and faithful, then conditions {\rm (v) -- (vii)} are equivalent to
\begin{itemize}
\item[(viii)]  $N\psi_{Jy}$ is an isomorphism in $\NN$ for each $ y  \in \obj \NN$. 
\end{itemize} 
\end{thm}
An immediate consequence is: 
\begin{cor}\label{c-a-e} The pair $NI\dashv MJ$ is an adjoint equivalence if and only if conditions  {\rm (i)} and  {\rm (v)} in Theorem \ref{main-c} are satisfied. 
\end{cor}

To prove Theorem \ref{main-c} we will use the following lemma.

\begin{lem}\label{iso-com}
Assume that we are given categories and functors $$P: \AA\to \CC, \, F:\CC \to \DD, \, G:\DD\to \CC, \,  Q: \BB \to \DD$$ such that 
$F \dashv G$. Then the comma category $FP\dn Q$ is isomorphic to the comma category $P \dn GQ$. 
\end{lem}

\begin{proof}
Let $\phi:\DD(Fx,y)\to \CC(x,Gy)$ for $x\in \obj\CC$ and $y\in\obj\DD$ denote the natural bijections implementing the adjunction $F\dashv G$. We define a map $R$ from $FP\dn Q$ to $P \dn GQ$ as follows:
set $R((a,b, f)) = (a, b, \phi(f)) \in \obj (P \dn GQ)$
for each $(a, b, f) \in \obj (FP\dn Q)$,
and set $R((k,h)) = (k,h)\colon (a, b, \phi(f)) \to (a', b', \phi(f'))$
for each $(k,h): (a,b, f) \to (a', b', f')$ in $FP\dn Q$. 
It is routine to check that $R$ is an isomorphism between the two comma categories.
\end{proof}

\emph{Proof of Theorem \ref{main-c}}. Let $x \in \obj \MM, y \in \obj \CC$. As $\sigma$ is the unit of   $I \dashv M$, the bijection $\phi:\CC(Ix,y)\to\MM(x,My)$ implementing this adjunction is given by $\phi(h)= (Mh)\circ \sigma_x$. Now, using Lemma \ref{iso-com} and its proof,  we have $$I \dn JNIx \, \simeq \, 1_{\MM} \dn MJNIx = \MM \dn MJNIx$$
under an isomorphism which sends $(x, \theta_{Ix})$ to $$(x, \phi(\theta_{Ix}))= (x, (M\theta_{Ix}) \circ \sigma_x) = (x, \eta_x)\,.$$ 
It follows that $(x,\theta_{Ix})$ is  final in $I\dn JNIx$ if and only if $(x, \eta_x)$ is final in the slice category $\MM \dn MJNIx$, that is, if and only if $\eta_x$ is an isomorphism. 

This shows that (i) is equivalent to (iii). Lemma  \ref{properties} gives that (ii) is equivalent to (iii).
If $I$ is full and faithful, then $\sigma$ is a natural isomorphism and the equivalence between (iii) and (iv) follows from equation (\ref{eq-u}). 
 Hence the first half is shown, and the second half follows in a dual way. 
\hfill $\square$ 

\begin{thm} \label{hyp-thm} Consider the following conditions:
\begin{itemize}
\item[(i)] $(Nx, \theta_x \circ \psi_x)$ is an inital object in $IMx \dn J$ for each $x \in \obj \CC$.
\item[(ii)] %the functor $N$ inverts $\psi$; that is, 
$N\psi_x$ is an isomorphism for  each $x \in \obj \CC$.
\item[(iii)] $(Mx, \theta_x \circ \psi_x)$ is a final object in $I \dn JNx $ for each $x \in \obj \CC$.
\item[(iv)] %the functor $M$ inverts $\theta$; that is, 
$M\theta_x$ is an isomorphism for  each $x \in \obj \CC$.
\end{itemize}
Then  {\rm (i)} $\Leftrightarrow$ {\rm (ii)} and {\rm (iii)} $\Leftrightarrow$ {\rm (iv)}. 

If these conditions are satisfied and  $I$ and $J$ are both full and faithful, then $NI\dashv MJ$ is an adjoint equivalence. 

On the other hand, if $NI\dashv MJ$ is an adjoint equivalence, then all four conditions above are equivalent; moreover, $I$ and $J$ are then both full and faithful whenever one of these four conditions holds. 
\end{thm}

\begin{proof} Let $x \in \obj \CC, \, y \in \obj \NN$. As $\theta$ is the unit of   $N \dashv J$, the bijection $\phi':\NN(Nx,y)\to\CC(x,Jy)$ implementing this adjunction is given by $\phi'(h)= (Jh)\circ \theta_x$. Further, using the properties of $\psi$ and $\theta$, it is not difficult to check that  $\theta_x \circ \psi_x = (JN\psi_x)\circ \theta_{IMx}\, $. 

Hence, using Lemma \ref{iso-com} and its proof,  we have $$  NIMx\dn \NN = NIMx \dn 1_{\NN} \,  \simeq \, IMx \dn J  \,$$
under an isomorphism which sends $(Nx, N\psi_x)$ to $$(Nx, \phi'(N\psi_x))= (Nx, (JN\psi_x) \circ \theta_{IMx}) = (Nx,  \theta_x\circ \psi_x)\,.$$ 
It follows that $(Nx, \theta_x  \circ  \psi_x)$ is  inital in $IMx \dn J$ if and only if $(Nx, N\psi_x)$ is initial in the slice category $  NIMx\dn \NN$, that is, if and only if $N\psi_x$ is an isomorphism. 
This shows that~(i) is equivalent to~(ii). The equivalence of~(iii) and~(iv) is dual.

If conditions~(i)--(iv)
%if {\rm (i)} or {\rm (ii)}, and {\rm (iii)} or {\rm (iv)} 
are satisfied and  $I$ and $J$ are both full and faithful, then $\sigma$ and $\rho$ are natural isomorphisms, and we see from equations (\ref{eq-u}) and (\ref{eq-cu}) that $\eta$ and $\epsilon$ are natural isomorphisms, hence that $NI\dashv MJ$ is an adjoint equivalence.  

Conversely, assume that $NI\dashv MJ$ is an adjoint equivalence. Then
$MJ$ is full and faithful (Corollary~\ref{c-a-e})
and $\eta_{Mx}$ is an isomorphism (Theorem~\ref{main-c}).
Since
\begin{align*}(MJN\psi_x)\circ \eta_{Mx} &= (MJN\psi_x)\circ M\theta_{IMx} \circ \sigma_{Mx} \\&= 
M\big((JN\psi_x)\circ \theta_{IMx}\big) \circ \sigma_{Mx}
= M(\theta_x\circ \psi_x)  \circ \sigma_{Mx}\\&= M\theta_x\circ M\psi_x  \circ \sigma_{Mx}
= M\theta_x
\end{align*}
for each object $x$ in $\CC$,
we see that $M\theta_x$ is an isomorphism if and only if $N\psi_x$ is an isomorphism. 
It follows that~(ii) is equivalent to~(iv), hence that all four conditions are equivalent. If one of them holds, then (ii) and (iv) hold, and  from equations (\ref{eq-u}) and (\ref{eq-cu}) we now see that  $\sigma$ and $\rho$ must be natural isomorphims; that is, $I$ and $J$ must both be full and faithful.
\end{proof}

\section{Reflective-coreflective equivalence}
\label{equivalence}

We now apply the general theory to a curious sort of equivalence between full subcategories, one reflective and the other coreflective. We have not been able to find this type of equivalence in the category-theory literature.

We let  $\MM$ and $\NN$ be full subcategories of a category $\CC$, with 
$\NN$ reflective and $\MM$ coreflective.
We will use the same notation as in the previous section, now with $I=\inc_\MM$ and $J=\inc_\NN$.
Note that $I$ and~$J$ are both full and faithful since we are assuming that $\MM$ and $\NN$ are full subcategories of $\CC$. 
Thus:
%This means that $\rho$ and $\sigma$ are natural isomorphisms.

\begin{notn}\label{above-notn}
\item $N:\CC\to\NN$ is a reflector, 

\item $\theta:1_\CC\to \inc_\NN\circ N$ is the unit of the adjunction $N\dashv \inc_\NN$,

\item $\rho : N \circ \inc_\NN \to 1_\NN$ is the counit of the adjunction $N\dashv \inc_\NN$,

\item $M:\CC\to\MM$ is a coreflector,

\item $\psi:\inc_\MM\circ M\to 1_\CC$ is the counit of the adjunction $\inc_\MM\dashv M$, and

\item $\sigma :  1_\MM\to M \circ \inc_\MM$ is the unit of the adjunction $\inc_\MM\dashv M$.
\end{notn}

Writing $N|_\MM=N \circ\inc_\MM$ and $M|_\NN = M \circ \inc_\NN$,
the following diagram illustrates this special situation:
\[
\xymatrix@C+30pt{
\MM \ar@<1ex>[r]^{\inc_\MM} \ar@/^2pc/[rr]^{N|_\MM}
&\CC \ar@<1ex>[r]^N \ar@<1ex>[l]^M
&\NN \ar@<1ex>[l]^{\inc_\NN} \ar@/^2pc/[ll]^{M|_\NN}
}
\]
By composition,  we have $N|_\MM\dashv M|_\NN$, with unit $\eta:1_\MM \to M|_\NN \circ N|_\MM$ and counit $\epsilon: N|_\MM \circ M|_\NN \to 1_\NN$.

%\medskip $ \eta_x:x\to MNx$ is given by $\eta_x= (M\theta_x) \circ \sigma_x$ for each $ x \in \obj \MM$;

%\medskip $ \epsilon_x: NMx\to x$ is given by $ \epsilon_x= \rho_x \circ (N\psi_x)$  for each $x \in \obj \NN$. 

%\medskip We are interested in conditions ensuring that $N|_\MM\dashv M|_\NN$ is an adjoint equivalence.  
%hence that $\MM$ and $\NN$ are $\CC$-equivalent according to the discussion of $\CC$-equivalence given in the introduction. 
In light of Theorem~\ref{main-c},
the following properties are of interest:

\begin{ppies}\label{two-way}\ 

\begin{enumerate}

\item[\, ({F})]
For each $x\in\obj\MM$, 
%not only is $(Nx, \theta_x)$ an initial object in $x\dn \NN$, but it is also a
$(x,\theta_x)$ is a {final} object in $\MM\dn Nx$; \\
in other words, $(x, \theta_x)$ is a universal morphism from
$\MM$ to $Nx$. 

\smallskip 
\item[({I})]
For each $y\in\obj\NN$, 
%not only is  $(Mx, \psi_x)$ a final object in $\MM\dn x$, but it is also an 
$(y,\psi_y)$ is an {initial} object in $My\dn\NN$; \\ 
in other words, $(y, \psi_y)$ is a universal morphism from $My$ to
$\NN$. 

\end{enumerate}
\end{ppies}

These conditions may be visualized by the following commutative diagrams:
\begin{align*}
\xymatrix{
\MM&\NN
\\
z \ar[dr] \ar@{-->}[d]_{!}
\\
x \ar[r]^{\theta_x} \ar[dr]
&Nx \ar@{-->}[d]^{!}
\\
&y
}
&&
\xymatrix{
\MM&\NN
\\
x \ar[dr] \ar@{-->}[d]_{!}
\\
My \ar[r]^{\psi_y} \ar[dr]
&y \ar@{-->}[d]^{!}
\\
&z
}
\end{align*}
In the left half, the top part is Property \ref{two-way}~(F),
and the bottom part is guaranteed by reflectivity of~$\NN$ in~$\CC$.
In the right half, the top part is guaranteed by coreflectivity of~$\MM$ in~$\CC$, 
while the bottom part is Property \ref{two-way}~(I).

As a consequence of Theorem \ref{main-c}, we immediately get:
\begin{thm}
\label{main}
%Suppose $\NN$ is a reflective subcategory
%and $\MM$ is a coreflective subcategory
%of a category $\CC$, 
%as in Notation~\ref{above-notn}.
%Then 
The following conditions are equivalent:

\begin{itemize}
\item[(i)] Property \ref{two-way}~{\rm (F)} holds.
\item[(ii)] $N|_\MM$ is full and faithful.
\item[(iii)] $\eta$ is a natural isomorphism.
\item[(iv)] For each $x \in \obj \MM,\,  M\theta_x:Mx \to MNx$  is an isomorphism \(in $\MM$ and therefore\) in $\CC$. 
\end{itemize}

Similarly, the following conditions are equivalent:

\begin{itemize}
\item[(v)] Property \ref{two-way}~{\rm (I)} holds.
\item[(vi)] $M|_\NN$ is full and faithful.
\item[(vii)] $\epsilon$ is a natural isomorphism.
\item[(viii)] For each $x \in \obj \NN,\,  N\psi_x:NMx \to Nx$  is an isomorphism \(in $\NN$ and therefore\) in $\CC$. 
\end{itemize}
\end{thm}

As a corollary to Theorem \ref{main} we get the following more precise version of~\cite[Proposition~2.1]{clda}:

\begin{cor}\label{main equivalence} 
The pair $N|_\MM\dashv M|_\NN$ is an adjoint equivalence 
if and only if 
Properties \ref{two-way}~{\rm (F)} and~{\rm (I)} hold.
\end{cor}

\section{``Maximal-normal'' type equivalence}
\label{boilerplate}

In this section we keep the hypotheses of Section~\ref{equivalence},
so $\NN$ is a full reflective subcategory of a category $\CC$
and $\MM$ is a full coreflective subcategory of $\CC$;
we also retain Notation~\ref{above-notn}. 
We further assume that the adjunction $N|_\MM \dashv M|_\NN$ is an adjoint equivalence; that is,  we assume that both Properties~\ref{two-way}~(F) and~(I) are satisfied (\cf~Corollary \ref{main equivalence}). Moreover, in order to capture the complete ``maximal-normal equivalence'' phenomenon exhibited by $C^*$-coactions in \cite{clda}, we also assume that the following condition is satisfied:

\begin{hyp}
\label{factor initial}
For each $x\in\obj\CC$,
$(Nx, \theta_x\circ\psi_x)$
is an initial object in the comma category $Mx\dn \NN$.
\end{hyp}
\begin{rem} \label{hyp-rem} Hypothesis \ref{factor initial} may be seen as a strengthening of 
Property~\ref{two-way}~(I). As we are also assuming that  Property~\ref{two-way}~(F) holds, 
it follows from Theorem \ref{hyp-thm} that  we could equally have assumed that $(Mx, \theta_x\circ\psi_x)$ is final in $\MM\dn Nx$ for each $x\in\obj\CC$
\end{rem} 

We immediately apply our new hypothesis:

\begin{prop}
\label{N factor}
$N\cong N|_\MM\circ M$.
\end{prop}

\begin{proof} 
It follows from Theorem \ref{hyp-thm} that $N\psi_x$ is an isomorphism for each $x\in\obj\CC$. This may also be seen directly:
in the diagram
\[
\xymatrix{
Mx \ar[r]^{\theta_{Mx}} \ar[d]_{\psi_x}
&NMx \ar@{-->}[d]^{N\psi_x}_{!}
\\
x \ar[r]_{\theta_x}
&Nx,
}
\]
both $(NMx, \theta_{Mx})$ and $(Nx, \theta_x\circ\psi_x)$ are initial in $Mx\dn\NN$,
so the unique morphism $N\psi_x$ in $\NN$
making the diagram commute in $\CC$ is an isomorphism.

Since $N$ is functorial and $\psi$ is a natural transformation, the composition
\[
N\psi:N|_\MM\circ M\to N
\]
is natural, and the result follows.
\end{proof}

%\begin{rem}\label{initialrem}
%The proof shows that Hypothesis \ref{factor initial} implies that $N\psi_x$ is an isomorphism in $\NN$ for each $x \in \obj\CC$. It is quite obvious from the diagram above that the converse implication is also true. Moreover, we see from Theorem \ref{main} that if $\NN$ is full, then 
%the property assumed in Hypothesis \ref{factor initial} is a strengthening of 
%Property~\ref{two-way}~(I) (where the condition is only imposed for $x\in\obj\NN$). 
%This can also be seen directly. 
%\end{rem}

\begin{cor}
\label{M factor}
$M\cong M|_\NN\circ N$.
\end{cor}

\begin{proof}
We could argue as in the proof of Proposition~\ref{N factor}, using Theorem~\ref{hyp-thm};
alternatively, using Proposition~\ref{N factor} directly we have
\begin{align*}
M|_\NN\circ N
&\cong M|_\NN\circ N|_\MM\circ M
\cong 1_\MM\circ M
=M.
\qedhere
\end{align*}
\end{proof}

We can deduce various consequences of the foregoing results; for example, \propref{N factor} and \corref{M factor} immediately give:

\begin{cor}\label{reflector}
$N|_\MM\circ M$ is a reflector of $\CC$ in $\NN$, 
and $M|_\NN\circ N$ is a coreflector of $\CC$ in $\MM$.
\end{cor}

Another consequence is:
\begin{cor}
\label{faithful}

The following conditions are equivalent:

\begin{enumerate}
\item For every $x\in\obj\CC$, $\psi_x:Mx\to x$ is an epimorphism in $\CC$.

\item $M$ is faithful.

\item $N$ is faithful.

\item For every $x\in\obj\CC$, $\theta_x:x\to Nx$ is a monomorphism in $\CC$. 
\end{enumerate}
\end{cor}

\begin{proof}
Since $\psi$ is the counit of  $ \inc_\MM \dashv M$
and $\theta$ is the unit of $N\dashv \inc_\NN$,
\lemref{properties} gives (i)$\Leftrightarrow$(ii) and (iii)$\Leftrightarrow$(iv).
Since $N\cong N|_\MM\circ M$ and $N|_\MM$ is an equivalence,
we have (ii)$\Leftrightarrow$(iii).
\end{proof}

\begin{rem}
\label{equivalence fail}
Even if we now  also assume  that 
%$\MM$ and $\NN$ are full and 
$\psi_x$ is an epimorphism for every $x\in\obj\CC$ (or, equivalently, $\theta_x$ is a monomorphism for every $x\in\obj\CC$), 
$N:\CC\to\NN$ itself can still fail to be an equivalence of categories: 
$N$ is faithful by Corollary \ref{faithful}, and
it is essentially surjective because the counit $\rho$
%$\rho:N\circ\inc_\NN\to 1_\NN$ 
is  a natural isomorphism. 
%(as $\NN$ is full).
But, although 
$\theta_x:x\to Nx$ is a monomorphism for all $x\in\CC$,
it is in general not an isomorphism for all $x$,
in which case $\theta_x$ will not be a split epimorphism,
and hence by \lemref{properties}
$N$ is not full. The point we are making here is that this is the only property of equivalences that $N$ can fail to possess.

Similarly, the coreflector $M:\CC\to\MM$ is then faithful (by \corref{faithful} again)
and essentially surjective (because the unit $\sigma$ is a natural isomorphism), 
but in general will not be full.
\end{rem}

\begin{rem}\label{analoghyp}
Hypothesis \ref{factor initial}, as well as the assumptions in Remark~\ref{equivalence fail},  are satisfied in 
%the case of 
the examples given in Section \ref{examples}.  But we don't know whether it is necessarily true that $\theta_x$ is an epimorphism for all $x\in\obj\CC$ and that $\psi_x$ is a monomorphism for all $x\in\obj\CC$, although these properties are satisfied in our examples.  
\end{rem}

\section{Examples}
\label{examples}

All our examples will involve $C^*$-algebras. We record here a few conventions which are not totally standard. By a homomorphism from a $C^*$-algebra (or just a $*$-algebra) into another, we will always mean a $*$-homomorphism. If  $X$ and $Y$ are  $*$-algebras, ${X\odot Y}$ will represent the
algebraic tensor product; 
if $X$ and $Y$ are $C^*$-algebras,
${X\otimes Y}$ will represent the minimal (\ie,~spatial) $C^*$-tensor
product~\cite[Chapter~6]{GM}.

\subsection{Coactions}
\label{coactions}

Our first --- in fact the ``original'' --- example of the ``maximal-normal'' equivalence involves coactions of groups on $C^*$-algebras.

\medskip Fix a locally compact Hausdorff group $G$.
Coactions of $G$ on $C^*$-algebras are dual to actions;
see \cite{nordfjordeid} for an introduction (including an exposition of the equivalence we will now describe),
or \cite[Appendix~A]{BE}.

 We will give here a development of the equivalence between maximal and normal coactions of $G$.
Most of the main results have appeared in the literature (mainly in \cite{clda}), but we will give an alternative development, with new proofs, and, in some cases, improvements upon existing results. We emphasize that these improvements arose from a close scrutiny of the underlying category theory.

One of our motivations for making this exposition essentially self-contained is that we find the existing literature on group coactions somehow unsatisfying, and in particular we sometimes find it inconvenient to dig specific results out of the currently available papers.

\medskip For the theory of coactions, we adopt the conventions of~\cite{BE}. All our coactions will be full and coaction-nondegenerate.

\begin{notnsub}\

\begin{enumerate}
\item \cs\ will denote the category whose objects are $C^*$-algebras 
and whose morphisms are nondegenerate homomorphisms into multiplier algebras, 
so that $\phi:A\to B$ in \cs\ means that $\phi:A\to M(B)$ is a homomorphism such that $\phi(A)B=B$.
For such a homomorphism, there is always a canonical extension $\bar\phi:M(A)\to M(B)$, 
and we have (for example) $\bar\psi\circ\bar\phi = \overline{\psi\circ\phi}$
when $\psi:B\to C$ in \cs.  

\item \coact\ will denote the category whose objects are coactions of $G$ on $C^*$-algebras, and whose morphisms are morphisms of~\cs\ that are
equivariant for the coactions, so that $\phi:(A,\delta)\to (B,\epsilon)$ in \coact\ means that
the diagram
\[
\xymatrix{
A \ar[r]^-\delta \ar[d]_\phi
&A\otimes C^*(G) \ar[d]^{\phi\otimes\id}
\\
B \ar[r]_-\epsilon
&B\otimes C^*(G)
}
\]
commutes in \cs.
\end{enumerate}
\end{notnsub}

In this example of the maximal-normal equivalence, the coreflective and reflective subcategories of~\cs\ are given by the maximal and normal coactions, respectively. To introduce these, it behooves us to say a few words about \emph{crossed-product duality} for $C^*$-dynamical systems:
for every coaction $(A,\delta)$ there is a \emph{crossed product} $C^*$-algebra, denoted $A\times_\delta G$, that encodes the representation theory of the coaction, and there is a \emph{dual action} $\what\delta$ of $G$ on $A\times_\delta G$ and a \emph{canonical surjection}
\[
\Phi:A\times_\delta G\times_{\what\delta} G\to A\otimes\KK(L^2(G)),
\]
where $\KK$ denotes the compact operators.
$(A,\delta)$ is
\emph{maximal} if $\Phi$
is an isomorphism, and 
\emph{normal} if $\Phi$ factors through an isomorphism of the \emph{reduced crossed product} by the dual action:
\[
\xymatrix{
A\times_\delta G\times_{\what\delta} G
\ar[r]^-\Phi
\ar[d]_\Lambda
&A\otimes \KK
\\
A\times_\delta G\times_{\what\delta,r} G
\ar@{-->}[ur]_\cong,
}
\]
where $\Lambda$ is the \emph{regular representation}.
The full subcategories of~\coact\ obtained by restricting to maximal or normal coactions will be denoted by \mcoact\ and \ncoact, respectively.

In practice, the following normality criterion is often useful: a coaction $(A,\delta)$ is normal 
if and only if $j_A:A\to M(A\times_\delta G)$ is injective,
where $j_A$ is the  ``$A$-part'' of the canonical covariant homomorphism $(j_A,j_G)$ of $(A,C_0(G))$
in the multiplier algebra of the crossed product. 
It is also useful to note that we can take
\[
j_A=(\id\otimes\lambda)\circ\delta,
\]
where $\lambda$ is the left regular representation of $G$.

\begin{notnsub}
For any object $(A,\delta)$ in \coact,
an object 
$((B,\epsilon),\phi)$
in the comma category $\mcoact\dn (A,\delta)$ will be denoted simply as a triple $(B,\epsilon,\phi)$, and similarly for the comma category $(A,\delta)\dn\ncoact$.
\end{notnsub}

Thus, to say $(B,\epsilon,\phi)$ is an object in $\mcoact\dn (A,\delta)$ means that $(B,\epsilon)$ is a maximal coaction and $\phi:B\to M(A)$ is a nondegenerate homomorphism that is $\epsilon-\delta$ equivariant, \ie, $\phi:(B,\epsilon)\to (A,\delta)$ in \coact.

\begin{defnsub}
\label{maximal normal def}
Let $(A,\delta)$ 
be a coaction.
\begin{enumerate}
\item A \emph{normalizer} of $(A,\delta)$ is 
an initial object
$(B,\epsilon,\eta)$ 
in $(A,\delta)\dn \ncoact$,
and we say $(B,\epsilon)$ is a \emph{normalization} of $(A,\delta)$.

\item A \emph{maximalizer} of $(A,\delta)$ is 
a final object $(B,\epsilon,\zeta)$ 
in $\mcoact\dn (A,\delta)$,
and we say $(B,\epsilon)$ is a \emph{maximalization} of $(A,\delta)$.
\end{enumerate}
\end{defnsub}

\begin{remsub}
Note that just knowing that $(B,\epsilon)$ is a normalization of $(A,\delta)$ doesn't uniquely determine a normalizer --- indeed, in general there will be many normalizers for a single normalization\footnote{and \emph{every} normalizer can be obtained from any particular one by pre- (alternatively, post-) composing with an automorphism of the respective coaction}. Our choice of terminology (particularly ``normalizer'') was designed to allow 
us to keep track of this distinction.
Similarly for maximalization.
\end{remsub}

\subsubsection*{Normalizations}

We need to know that normalizations exist:

\begin{propsub}
[{\cite[Proposition~2.6]{fullred}}]
\label{normal exist}
If $(A,\delta)$ is a coaction, 
then 
$j_A:(A,\delta)\to (j_A(A),\ad j_G)$ is a normalizer.
\end{propsub}

In the above proposition, we've committed a mild abuse of notation: by our earlier use of the notation ``$\ad$'', $\ad j_G$ would refer to an inner coaction, for example on $A\times_\delta G$; here of course we are using the same notation for the restriction of $\bar{\ad j_G}$ to $j_A(A)$. 
Moreover, we should formally have said ``$(j_A(A),\ad j_G,j_A)$ is a normalizer''. 
We will from now on sometimes be sloppy and refer to an object $(y,f)$ in a comma category just by the morphism $f$.

\begin{proof}
\corref{unitary} tells us that 
$j_A$ is a morphism of $(A,\delta)$ to the normal coaction $(j_A(A),\ad j_G)$, and by construction $j_A$ is surjective.
Let $\phi:(A,\delta)\to (B,\epsilon)$ in \coact\ with
$(B,\epsilon)$ normal. We need to know that there is a unique morphism $\rho$ in \coact\ making the diagram
\[
\xymatrix{
(A,\delta) \ar[r]^-{j_A} \ar[dr]_\phi
&(j_A(A),\ad j_G) \ar@{-->}[d]^\rho_{!}
\\
&(B,\epsilon)
}
\]
commute. It suffices to show that $\ker j_A\subset \ker\phi$.
By functoriality of crossed products, we have a commutative diagram
\[
\xymatrix{
A \ar[r]^-{j_A} \ar[d]_\phi
&A\times_\delta G \ar[d]^{\phi\times G}
\\
B \ar[r]_-{j_B}
&B\times_\epsilon G,
}
\]
so that
\begin{align*}
\ker j_A
&\subset \ker\bigl((\phi\times G)\circ j_A\bigr)
=\ker\bigl(j_B\circ\phi)
=\ker \phi
\end{align*}
because $j_B$ is injective.
\end{proof}

Upon examining the above particular normalizer, we discern a hidden property:

\begin{corsub}
\label{normalizer surjective}
Every normalizer is surjective.
\end{corsub}

\begin{proof}
This follows immediately from the following two observations: it is true for the particular normalizer in \propref{normal exist}, and all normalizers are isomorphic by universality of initial objects.
\end{proof}

In the following characterization of normalizations, the proof of the converse direction is essentially due to Fischer \cite[Lemma~4.2]{fischer:quantum} (see also \cite[Lemma~2.1]{maximal} --- the hypothesis in \cite{maximal} that the homomorphisms map into the $C^*$-algebras themselves rather into the multipliers is not used in the proof of~\cite[Lemma~2.1]{maximal}). We say ``essentially'' regarding \cite{fischer:quantum} because Fischer doesn't explicitly address equivariance.

\begin{propsub}
\label{normal universal}
An object $(B,\epsilon,\eta)$ of $(A,\delta)\dn \ncoact$ is a normalizer if and only if the morphism
\[
\eta\times G:A\times_\delta G\to B\times_\epsilon G
\]
in \cs\ is an isomorphism.
\end{propsub}

\begin{proof}
First assume that $(B,\epsilon,\eta)$ is a normalizer.
By \lemref{recognize crossed product}, to see that $\eta\times G:A\times_\delta G\to B\times_\epsilon G$ is an isomorphism it suffices to show that $(B\times_\epsilon G,j_B\circ\eta,j_G)$ is a crossed product of $(A,\delta)$.
Since $\eta$ is surjective by \corref{normalizer surjective}, $B\times_\epsilon G$ is generated by $j_B\circ\eta(A)j_G(C_0(G))$. Thus
by \lemref{generated} it suffices to show that every covariant homomorphism $(\pi,\mu)$ of $(A,\delta)$ factors through $(j_B\circ\eta,j_G)$.
By universality there is a unique morphism $\rho$ in \coact\ making the diagram
\[
\xymatrix{
(A,\delta) \ar[r]^-\eta \ar[dr]_\pi
&(B,\epsilon) \ar@{-->}[d]^\rho_{!}
\\
&(C,\ad\mu)
}
\]
commute. Then by \lemref{covariant}, $(\rho,\mu)$ is a covariant homomorphism of $(B,\epsilon)$ in $M(C)$, and
the morphism $\rho\times\mu:B\times_\epsilon G\to C$ in \cs\ satisfies
\begin{align*}
(\rho\times\mu)\circ j_B\circ\eta
&=\rho\circ\eta
=\pi,
\end{align*}
and of course
\[
(\rho\times\mu)\circ j_G=\mu.
\]

Conversely, suppose $\eta\times G$ is an isomorphism,
and let $(C,\gamma,\phi)$ be an object in $(A,\delta)\dn \ncoact$. We need to show that there is a unique morphism $\psi$ in \coact\ making the diagram
\[
\xymatrix{
(A,\delta) \ar[r]^\eta \ar[dr]_\phi
&(B,\epsilon) \ar@{-->}[d]^\psi_{!}
\\
&(C,\gamma)
}
\]
commute.
It suffices to observe that
 $\ker\phi\supset\ker\eta$, since
\begin{align*}
j_C\circ\phi
&=(\phi\times G)\circ j_A
\\&=(\phi\times G)\circ(\eta\times G)\inv\circ(\eta\times G)\circ j_A
\\&=(\phi\times G)\circ(\eta\times G)\inv\circ j_B\circ\eta
\end{align*}
and $j_C$ is injective.
\end{proof}

\medskip \begin{remssub}
(i)~For the first half of the above proof, we could have alternatively argued\footnote{and in fact there is some redundancy in the results presented here} as in \propref{maximal universal} below: note that by \corref{j_A crossed} and \lemref{normal exist} there is at least one normalizer $(C,\gamma,\sigma)$ for which $\sigma\times G$ is an isomorphism, and since any two normalizers are isomorphic it follows that $\eta\times G$ is also an isomorphism.

\smallskip (ii)~Since \corref{j_A crossed} shows that $j_A\times G$ is an isomorphism, the above proposition implies that $(j_A(A),\ad j_G,j_A)$ is a normalizer, giving an independent proof of~\lemref{normal exist}.
\end{remssub}

%\newpage 
\begin{notnsub}
\label{choose normalizer}
For every coaction $(A,\delta)$ we make the following choice of 
normalizer
$q^n:(A,\delta)\to (A^n,\delta^n)$:
\begin{itemize}
\item $A^n=A/\ker j_A$;

\item $\delta^n$ is the unique coaction of $G$ on $A^n$
corresponding to the coaction $\ad j_G$
under the canonical isomorphism $A^n\cong j_A(A)$;

\item $q^n=q^n_{(A,\delta)}:A\to A^n$ is the quotient map.
\end{itemize}
\end{notnsub}

Thus it follows from \propref{normal exist} 
that there is a unique functor $\nor:\coact\to\ncoact$ that
takes each object $(A,\delta)$ to $(A^n,\delta^n)$
and is a left adjoint to the inclusion functor,
so that
\ncoact\ is a reflective subcategory of~\coact\
and $\nor$ is a reflector,
with unit $q^n$.
Moreover, by our construction we can identify the normalization of every normal coaction with itself, 
so that the counit of this reflector is the identity transformation on the identity functor on the subcategory \ncoact.
What the normalization functor does to morphisms is characterized as follows: if $\phi:(A,\delta)\to (B,\epsilon)$ in \coact, then the normalization 
of $\phi$ is the unique morphism 
$\phi^n$ 
in \ncoact\footnote{indeed, unique in \coact, since the subcategory \ncoact\ is full} making the diagram
\[
\xymatrix{
(A,\delta) \ar[r]^{q^n} \ar[d]_\phi
&(A^n,\delta^n) \ar[d]^{\phi^n}
\\
(B,\epsilon) \ar[r]_{q^n}
&(B^n,\epsilon^n)
}
\]
commute.

\subsubsection*{Maximalizations}

The existence of maximalizations is established in \cite[Theorem~3.3]{maximal} and \cite[Theorem~6.4]{fischer:quantum}.
The construction in \cite{maximal} is noncanonical (involving a choice of minimal projection in the compacts), while Fischer's construction in \cite{fischer:quantum} is canonical (involving an appropriate relative commutant of the image of $\KK$ in the multipliers of the double crossed product).
However, having a specific formula for maximalizations
has not turned out to be particular useful,
and in fact
from a categorical perspective is clearly deprecated.
In certain situations where the cognoscenti ``know'' what the maximalization should be, we'll be careful to say ``a maximalization'' (or ``a maximalizer''). For instance, if $(A,G,\alpha)$ is an action, then the regular representation
\[
\Lambda:(A\times_\alpha G,\what\alpha)\to (A\times_{\alpha,r} G,\what\alpha^n)
\]
is \emph{the} normalization of the dual coaction on the full crossed product, but is only \emph{a} maximalization of the dual coaction on the reduced crossed product. The point is that, given only the coaction $(A\times_{\alpha,r} G,\what\alpha^n)$, we can't reconstruct what the action $(A,\alpha)$ was, and so we can't reconstruct the full crossed product. Again, Fischer tells us how to pick a canonical maximalization, but we will not do that.

As with normalizers, we have an automatic surjectivity for maximalizers:

\begin{lemsub}
\label{maximalizer surjective}
Every maximalizer is surjective.
\end{lemsub}

\begin{proof}
The argument is similar to \corref{normalizer surjective}: the maximalizers constructed in both \cite{maximal} and \cite{fischer:quantum} are surjective, and by universality all maximalizers are isomorphic.
\end{proof}

\begin{propsub}
\label{maximal universal}
An object $(B,\epsilon,\zeta)$ of $\mcoact\dn (A,\delta)$ is a maximalizer if and only if the morphism
\[
\zeta\times G:B\times_\epsilon G\to A\times_\delta G
\]
in \cs\ is an isomorphism.
\end{propsub}

\begin{proof}
First suppose that $(B,\epsilon,\zeta)$ is a maximalizer.
To see that $\zeta\times G$ is an isomorphism, it will suffice to know that there is \emph{at least one} maximalizer $(C,\gamma,\sigma)$ for which $\sigma\times G$ is an isomorphism;
for example, this holds for the constructions of maximalizers in both \cite{fischer:quantum} and \cite{maximal}.
By universality of maximalizers there is an isomorphism
\[
\theta:(B,\epsilon,\zeta)\to (C,\gamma,\sigma).
\]
Then in particular $\theta$ gives an isomorphism $(B,\epsilon)\cong (C,\gamma)$ of coactions, and we have a commuting diagram
\[
\xymatrix{
B\times_\epsilon G \ar[r]^-{\theta\times G}_\cong \ar[dr]_{\zeta\times G}
&C\times_\gamma G \ar[d]^{\sigma\times G}_\cong
\\
&A\times_\delta G.
}
\]
Thus $\zeta\times G$ is an isomorphism.

Conversely, suppose that $\zeta\times G$ is an isomorphism, and let $(C,\gamma,\phi)$ be an object in $\mcoact\dn (A,\delta)$. We need to show that there is a unique morphism $\psi$ in \coact\ making the diagram
\[
\xymatrix{
(C,\gamma) \ar[dr]^\phi \ar@{-->}[d]_\psi^{!}
\\
(B,\epsilon) \ar[r]_-\zeta
&(A,\delta)
}
\]
commute.

Consider the diagram
\[
\xymatrix{
C \ar[rr]^{\id\otimes 1} \ar@{-->}[dd]_\psi^{!} \ar[dr]^\phi
&&
C\otimes\KK \ar'[d][dd]^\sigma \ar[dr]^{\phi\otimes\id}
&&
C\times G\times G \ar[ll]_{\Phi_C}^\cong \ar[dr]^{\phi\times G\times G}
\\
&A \ar[rr]^(.3){\id\otimes 1}
&&
A\otimes\KK
&&
A\times G\times G \ar[ll]^{\Phi_A}
\\
B \ar[rr]^{\id\otimes 1} \ar[ur]_\zeta
&&
B\otimes\KK \ar[ur]_{\zeta\otimes\id}
&&
B\times G\times G \ar[ll]^{\Phi_B}_\cong \ar[ur]_{\zeta\times G\times G}^\cong
}
\]
in \cs,
where we define
\[
\sigma
=\Phi_B
\circ(\zeta\times G\times G)\inv
\circ(\phi\times G\times G)
\circ(\Phi_C)\inv,
\]
so that the diagram (without $\psi$) commutes.
We must show that there is a unique morphism $\psi$ making the left triangle commute, and moreover that $\psi$ is $\gamma-\epsilon$ equivariant.

Note that by crossed-product duality theory we have
\[
\bar\sigma\bigm|_{\bigl(1_{M(C)}\otimes\KK\bigr)}=1_{M(B)}\otimes\id_{\KK}.
\]
It follows that $\bar\sigma$ maps $C\otimes 1_{M(\KK)}$ into (the canonical image in $M(B\otimes \KK)$ of) $M(B)\otimes 1_\KK$. Thus there is a unique homomorphism $\psi:C\to M(B)$ such that
\[
\sigma=\psi\otimes 1_{M(\KK)},
\]
and moreover $\psi$ is nondegenerate since $\sigma$ is.

For the equivariance of $\psi$, note that, again by the general theory of crossed-product duality, the morphism $\sigma$ is $(\gamma\otimes_*\id)-(\epsilon\otimes_*\id)$ equivariant,
where by ``$\otimes_*$'' we mean that, in order to have an honest coaction, tensoring with $\id_\KK$ must be followed by a switching of the last two factors in the triple tensor product, so that, for example,
\[
\gamma\otimes_*\id=(\id\otimes\Sigma)\circ(\gamma\otimes\id),
\]
where
\[
\Sigma:C^*(G)\otimes\KK\to \KK\otimes C^*(G)
\]
is the flip isomorphism.
Thus we have
\begin{align*}
(\id\otimes\Sigma)\circ\bigl((\epsilon\circ\psi)\otimes\id\bigr)
&=(\id\otimes\Sigma)\circ(\epsilon\otimes\id)\circ(\psi\otimes\id)
\\&=(\epsilon\otimes_*\id)\circ\sigma
\\&=(\sigma\otimes\id)\circ(\gamma\otimes_*\id)
\\&=(\psi\otimes\id\otimes\id)\circ(\id\otimes\Sigma)\circ(\gamma\otimes\id)
\\&=(\id\otimes\Sigma)\circ(\psi\otimes\id\otimes\id)\circ(\gamma\otimes\id)
\\&=(\id\otimes\Sigma)\circ\Bigl(\bigl((\psi\otimes\id)\circ\gamma\bigr)\otimes\id\Bigr),
\end{align*}
so because $\id\otimes\Sigma$ is injective we have
$(\epsilon\circ\psi)\otimes\id=\bigl((\psi\otimes\id)\circ\gamma\bigr)\otimes\id$,
and therefore $\epsilon\circ\psi=(\psi\otimes\id)\circ\gamma$.
\end{proof}

\begin{remssub}
(i)~The above property of giving isomorphic crossed products was in fact the definition of maximalization given in \cite{fischer:quantum} and \cite{maximal}. In \defnref{maximal normal def} we use the universal property as the definition because it can be stated completely within the original category.
Also, in \cite{fischer:quantum} and \cite{maximal} the property involving isomorphic crossed products was only shown to imply the universal property of~\defnref{maximal normal def}, not that the two properties are in fact equivalent, as proved above.

(ii)~In \cite[Lemma~3.2]{clda} it is shown that in fact \emph{any} morphism $\phi:(A,\delta)\to (B,\delta)$ in \coact\ for which $\phi\times G$ is an isomorphism is surjective. Moreover, in \cite[Proposition~3.1]{clda} it is shown that $\phi$ is in fact an isomorphism if either $(A,\delta)$ is normal or $(B,\epsilon)$ is maximal.
\end{remssub}

\begin{notnsub}
For every coaction $(A,\delta)$ we assume that a 
maximalizer
\[
q^m:(A^m,\delta^m)\to (A,\delta)
\]
has been chosen, with the proviso that if $(A,\delta)$ is maximal then
\[
(A^m,\delta^m)=(A,\delta)\and q^m=\id_A.
\]
\end{notnsub}

Thus it follows 
that there is a unique functor $\maxi:\coact\to\mcoact$ that
takes each object $(A,\delta)$ to $(A^m,\delta^m)$
and is a right adjoint to the inclusion functor,
so that
\mcoact\ is a coreflective subcategory of~\coact\
and $\maxi$ is a coreflector,
with counit $q^m$.
Moreover, since we have chosen the coreflector to do nothing to maximal coactions, the unit of this coreflector is the identity transformation on the identity functor on the subcategory \mcoact.
What the maximalization functor does to morphisms is characterized as follows: if $\phi:(A,\delta)\to (B,\epsilon)$ in \coact, then the maximalization of $\phi$ is the unique morphism $\phi^m$ in \mcoact\footnote{indeed, unique in \coact, since the subcategory \mcoact\ is full} making the diagram
\[
\xymatrix{
(A^m,\delta^m) \ar[r]^{q^m} \ar[d]_{\phi^m}
&(A,\delta) \ar[d]^\phi
\\
(B^m,\epsilon^m) \ar[r]_{q^m}
&(B,\epsilon)
}
\]
commute.

We have now defined a 
coreflector $\maxi:\coact\to\mcoact$
and a 
reflector $\nor:\coact\to\ncoact$.
The following two lemmas show that $\maxi$ and $\nor$ satisfy 
Properties \ref{two-way}~(F) and~(I).

\begin{lemsub}
\label{q^m normalizer}
Let $(A,\delta)$ be a normal coaction.
Then not only is 
$q^m:(A^m,\delta^m)\to (A,\delta)$
a maximalizer,
it is also a 
normalizer.
\end{lemsub}

Thus, not only is 
$q^m$
final in $\mcoact\dn (A,\delta)$, it is also initial in $(A^m,\delta^m)\dn \ncoact$.

\begin{proof}
$(A,\delta)$ is normal, 
so $(A,\delta,q^m)$ is an object of $(A^m,\delta^m)\dn\ncoact$.
Since $q^m\times G:A^m\times_{\delta^m} G\to A\times_\delta G$ is an isomorphism,
the result follows from \propref{normal universal}.
\end{proof}

\begin{lemsub}
\label{q^n maximalizer}
Let $(A,\delta)$ be a maximal coaction.
Then not only is 
$q^n:(A,\delta)\to (A^n,\delta^n)$
a 
normalizer,
it is also a 
maximalizer.
\end{lemsub}

Thus, not only is 
$q^n$
initial in $(A,\delta)\dn\ncoact$, it is also final in $\mcoact\dn (A^n,\delta^n)$.

\begin{proof}
The proof is similar to the above:
$(A,\delta)$ is maximal, 
so $(A,\delta,q^n)$ is an object of $\mcoact\dn(A^n,\delta^n)$.  
Since $q^n\times G:A\times_\delta G\to A^n\times_{\delta^n} G$ is an isomorphism,
the result follows from \propref{maximal universal}.
\end{proof}

\begin{corsub}
[{\cite[Theorem~3.3]{clda}}] 
$\nor|_{\mcoact} \dashv \maxi|_{\ncoact}$ is an adjoint equivalence.
In particular, 
$\nor|_{\mcoact}:\mcoact\to\ncoact$
is an equivalence, and
$\maxi|_{\ncoact}$ is a quasi-inverse.
\end{corsub}

\begin{proof}
This follows immediately from the above two lemmas and
Corollary \ref{main equivalence}.
\end{proof}

We now show that $\maxi$ and $\nor$ satisfy 
the extra property recorded in \hypref{factor initial} 
and its dual analog mentioned in Remark \ref{analoghyp}.

\begin{lemsub}
Let $(A,\delta)$ be a coaction.
Then 
$q^n\circ q^m:(A^m,\delta^m)\to (A^n,\delta^n)$
is both
a 
normalizer
and a 
maximalizer.
\end{lemsub}

The notation in the above lemma is unambiguous, but just to be clear: in the composition $q^n\circ q^m$ the maps are
\[
q^n:A\to A^n\and q^m:A^m\to A.
\]

\begin{proof}
We only prove the first statement; the second one is similar.
The coaction $(A^n,\delta^n)$ is normal, and
$q^n\circ q^m$ is an equivariant surjection, since $q^n$ and $q^m$ are. By functoriality of crossed products, we have
\[
(q^n\times G)\circ (q^m\times G)=(q^n\circ q^m)\times G
:A^m\times_{\delta^m} G\to A^n\times_{\delta^n} G.
\]
Since both $q^n\times G$ and $q^m\times G$ are isomorphisms, so is $(q^n\circ q^m)\times G$.
Thus 
$q^n\circ q^m$ is a normalizer,
by
\propref{normal universal}.
\end{proof}

The following three consequences may be new, and result from careful consideration of the categorical perspective:

\begin{corsub}
\label{nor factor}
$\nor\cong \nor|_{\mcoact}\circ\maxi$
and
$\maxi\cong\maxi|_{\ncoact}\circ\nor$.
\end{corsub}

\begin{proof}
This now follows immediately from 
\propref{N factor} and \corref{M factor}
\end{proof}

\begin{corsub}
$\maxi$ and $\nor$ are both faithful.
\end{corsub}

\begin{proof}
Since maximalizers are surjective, they are epimorphisms in \coact, so this follows immediately from 
\corref{faithful}.
\end{proof}

\begin{corsub}
\label{q^n monomorphism}
If $(A,\delta)$ is a coaction, then the map $q^n:(A,\delta)\to (A^n,\delta^n)$, and hence every normalizer of $(A,\delta)$, is a monomorphism in~\coact.
\end{corsub}

\begin{proof}
The first part follows immediately from 
\lemref{properties},
and then the second part follows since any two normalizers of $(A,\delta)$ are isomorphic in \coact.
\end{proof}

\begin{remsub}
The preceding corollary surprised us at first.
Certainly $q^n:A\to A^n$ is not generally a monomorphism in \cs, because it would then have to be injective,
which it frequently fails to be --- for example, if $G$ is a locally compact group then
$q^n:(C^*(G),\delta_G)\to (C^*_r(G),\delta^n_G)$ 
is the integrated form of the regular representation, 
which is noninjective if $G$ is nonamenable.

It is instructive to repeat 
\remref{equivalence fail}
in the present context:
let us examine 
exactly
how
the functor
$\nor:\coact\to\ncoact$ itself fails to be an equivalence of categories.
We have seen that $\nor$ is faithful, and it is not only essentially surjective, as any reflector in a full reflective subcategory must be, but in our case is actually \emph{surjective} on objects, because we have insisted that the reflector satisfy
\[
\nor|_{\ncoact}=1|_{\ncoact}.
\]
Thus the unit $q^n$, although it is always both an epimorphism  and a monomorphism in \coact, is not generally an isomorphism. In particular, it is not a \emph{split} epimorphism, so by
\lemref{properties}
$\nor$ is not full, and that is the only property of equivalences that it fails to possess.
\end{remsub}

\subsection{Compact quantum groups}
\label{quantum}

Our next example of the reflective-coreflective equivalence involves compact quantum groups
as defined by S.L.~Woronowicz~\cite{Wo1, Wo2} -- see also~\cite{KT, MvD, BMT}.  
For the ease of the reader, we begin by recalling  some basic facts about these objects. 

A \emph{compact quantum group} $(A, \Delta )$ consists of
a unital $C^*$-algebra $A$ (with unit $1=1_A$) 
and a unital homomorphism $\D: A \rightarrow A \otimes A$ (called the co-multiplication)
satisfying
\[ (\id \otimes \D)\circ \D = (\D \otimes \id)\circ \D, \]
and such that the linear spans of $(1 \otimes A) \D(A)$ 
and $(A \otimes 1) \D(A)$ are each dense in $A \otimes A$. 
For any compact quantum group $(A,\Delta)$, 
there exists a unique state $h=h_A$ on $A$, 
called the {\it Haar state} of $(A, \Delta )$, 
which satisfies
\[ (h \otimes \id) \circ\D = (\id \otimes h) \circ\D =  h(\cdot) 1. \]
(These conditions are known, respectively, as \emph{left-} and \emph{right-invariance} of~$h$.)

By a \emph{Hopf $*$-subalgebra} $\AA$ of $(A, \D )$  we mean a Hopf $*$-algebra~$\AA$ 
which is a unital $*$-subalgebra of~$A$ 
with co-multiplication given by restricting the co-multiplication $\D$ from $A$ to $\AA$. 
(As a Hopf $*$-algebra, $\AA$ has a co-unit and a co-inverse, 
but they won't play any role in our discussion).

Any compact quantum group $(A, \D )$ has a canonical dense Hopf $*$-subalgebra $\AA$, 
called the \emph{associated Hopf $*$-algebra} of $(A, \D)$;  
$\AA$ is the linear span of the matrix entries 
of all finite dimensional co-representations of $(A, \D )$. 
Here, when $n \in \N$, an \emph{$n$-dimensional co-representation} of $\AD$ means 
a unitary matrix $U=(u_{ij}) \in M_n(A)$ satisfying 
\[\D(u_{ij}) = \sum_{k=1}^{n} u_{ik} \otimes u_{kj}, \quad i, j = 1, \ldots, n\,.\] 

The associated Hopf $*$-algebra of  $(A, \D)$
is the unique dense Hopf $*$-subalgebra of $(A, \D)$ 
(see the appendix of \cite{BMT} for a proof).
It is known (\cf~\cite{Wo2}) that the
Haar state of $(A, \D )$ is faithful on $\AA$, but 
not on $A$ in general. 

Now let $\AD$ and $(B,\D')$ be compact quantum groups 
with associated Hopf $*$-algebras $\AA$ and $\BB$, respectively.  
A \emph{quantum group morphism} from  $\AD$ to $(B,\D')$ is 
a unital homomorphism ${\pi : A\to B}$ satisfying
$$\D'\circ \pi={(\pi\otimes \pi)\circ\D}.$$ 
Using this equation, one easily sees that 
if $U=(u_{ij}) \in M_n(A)$ is a co-representation of $\AD$, 
then $V= (\pi(u_{ij})) \in M_n(B)$ is  a co-representation of $(B, \D')$. 
It follows that $\pi(\AA)\subseteq\BB$.

The obvious category whose objects are compact quantum groups 
and morphisms are quantum group morphisms 
has too many morphisms for our purposes;
our category $\CC$ will be obtained by considering only those morphisms 
satisfying a certain natural condition. 
The following lemma illustrates two ways of describing this condition. 
Whenever $\pi$ satisfies one of these equivalent conditions, 
we will say that $\pi$ is a \emph{strong} quantum group morphism.

\begin{lemsub}\label{fundlem} 
Let $\pi$ be a quantum group morphism from  $\AD$ to $(B,\D')$. 
Then the restriction of $\pi$ to $\AA$ is injective 
if and only if
$h_A = h_B \circ \pi$. 
\end{lemsub} 

\begin{proof} 
Assume first that $\pi_{|\AA} : \AA \to B$ is injective,
and set  $h' = h_B \circ \pi$. 
We will show that $h' = h_A$. 

Let $a \in \AA$. 
Then we have
\begin{align*}
\pi(h'(a) \, 1_A) 
&= h'(a) \, 1_B 
= h_B(\pi(a))\, 1_B 
= \big((h_B\otimes \id_B)\circ \D'\big)(\pi(a))\\
&= \big((h_B\otimes \id_B)\circ (\pi \otimes \pi)\big)(\D(a)))
= \pi\big((h' \otimes \id_A)(\D(a))\big).
\end{align*}
As $\D(a) \in \AA \odot \AA$, we have $(h' \otimes \id_A)(\D(a)) \in \AA$. 
The injectivity of $\pi$ on $\AA$ then implies that $(h' \otimes \id_A)(\D(a))= h'(a)\, 1_A$. 

In the same way, one gets $(\id_A \otimes h')(\D(a))= h'(a)\, 1_A$. 
Hence the state $h'$ is left- and right-invariant on $\AA$, and therefore also on $A$, 
by density of $\AA$ and  continuity of the involved maps. 
By the uniqueness property of the Haar state on $A$, 
it follows that $h'=h_A$, as desired. 

Assume now that $h_A = h_B\circ \pi$. 
To show that $\pi$ is injective on $\AA$, consider $a \in \AA$ satisfying $\pi(a) = 0$. 
Then we have
\[
h_A(a^*a)= h_B(\pi(a^*a)) = h_B(\pi(a)^*\pi(a))= h_B(0)=0
\]
But $h_A$ is faithful on $\AA$, so $a = 0$. 
\end{proof}

It is straightforward to check that the usual composition (as maps) 
of two strong quantum group morphisms, whenever it makes sense, 
is again a strong quantum group morphism. 
The following definition is therefore meaningful. 

\begin{defnsub} 
The category $\CC$ has  compact quantum groups as objects. 
Its morphisms are strong quantum group morphisms. 
Composition of morphisms is given by usual composition of maps, 
while the identity morphisms are just the identity maps.  
\end{defnsub}

\subsubsection*{Reduced compact quantum groups}

Let $\AD$ be a compact quantum group with associated Hopf $*$-algebra $\AA$. 
The left kernel $N_A= \{a \in A \mid h_A(a^*a)=0\}$ of~$h_A$ 
is then known to be a two-sided ideal of $A$. 
Set $A_{\rm r}=A/N_A$ and let $\t_A$ denote the quotient map from $A$ onto $A_{\rm r}$. 

The $C^*$-algebra $A_{\rm r}$
can be made into a compact quantum group $(A_{\rm r},\D_{\rm r})$, 
called  the \emph{reduced quantum group} of~$(A,\Delta)$ 
(\cf~\cite{Wo1} and \cite[Section 2]{BMT} for details):

The co-multiplication $\D_{{\rm r}}$ is determined by the equation 
$\D_{{\rm r}}\circ \t_A=(\t_A\otimes \t_A)\circ  \D$.
The quotient map $\t_A$ is injective on~$\AA$ 
and $\t_A(\AA)$ is the Hopf $*$-algebra of $(A_{\rm r},\D_{{\rm r}})$. 
In particular, this means that $\t_A$ is a morphism in $\CC$ 
from $\AD$ to $(A_{\rm r},\D_{\rm r})$. 
Moreover, the Haar state of $(A_{\rm r},\D_{{\rm r}})$  is faithful 
and is the unique state $\hr$ of $A_{\rm r}$ such that $h_A=\hr\circ \t_A$. 

We will say that $\AD$ is \emph{reduced} whenever $h_A$ is faithful on $A$, 
\ie~whenever  $N_A = \{ 0\}$, in which case we will identify $(A_{\rm r},\D_{{\rm r}})$ with $\AD$. Clearly, the  reduced quantum group
of any $(A,\Delta)$ is reduced.

\begin{defnsub} The category $\RR$ is the full subcategory of $\CC$ whose objects are reduced compact quantum groups.
\end{defnsub}

To see that  reduction gives a functor $R$ from $\CC$ to $\RR$, we will use the following lemma.

\begin{lemsub} \label{redulem} 
Let $\pi$ be a strong quantum group morphism from  $\AD$ to $(B,\D')$. 
Then there exists a unique strong quantum group morphism 
$\pi_{\rm r}$ from  $(A_{\rm r},\D_{\rm r})$ to $(B_{\rm r} ,\D'_{\rm r})$ 
such that $\pi_{\rm r}\circ \t_A=\t_B \circ \pi $, 
that is, making the diagram
\[
\xymatrix{
\AD \ar[r]^{\t_A} \ar[d]_\pi
&(A_{\rm r},\D_{\rm r}) \ar[d]^{\pi_{\rm r}}
\\
(B,\D') \ar[r]_{\t_B}
&(B_{\rm r}, \D'_{\rm r})
}
\]
commute.
\end{lemsub}

\begin{proof} 
As $h_A(a^*a)=  h_B(\pi(a)^*\pi(a))$ for $a \in A$, 
it follows readily  that $\pi(N_A) \subseteq N_B$.%
\footnote{This is not necessarily true if $\pi$ is not strong. 
Consequently, Lemma \ref{redulem} is not true for general quantum group morphisms.} 
Hence we may define $\pi_{\rm r}: A_{\rm r} \to B_{\rm r}$ by  
\[
\pi_{\rm r}(\t_A(a))=\t_B(\pi(a)), \quad a \in A.
\]
It is easy to check that 
$\D'_{\rm r} \circ \pi_{\rm r}\circ \t_A  
= (\pi_{\rm r} \otimes \pi_{\rm r})\circ \D_{\rm r}\circ \t_A$. 
This implies that $\pi_{\rm r}$ is a quantum group morphism  
from  $(A_{\rm r},\D_{\rm r})$ to $(B_{\rm r} ,\D'_{\rm r})$ 
satisfying $\pi_{\rm r}\circ \t_A=\t_B \circ \pi $.

Letting $h_{\rm r} $ and $h'_{\rm r} $ denote 
the respective Haar states of $(A_{\rm r},\D_{\rm r})$ and $(B_{\rm r}, \D'_{\rm r})$, 
we have
\[
(h'_{\rm r} \circ \pi_{\rm r})\circ \t_A  = h'_{\rm r}\circ \t_B \circ \pi = h_B \circ \pi = h_A.
\]
From the uniqueness property of $h_{\rm r}$, we get $h_{\rm r} = h'_{\rm r} \circ \pi_{\rm r}$, 
so $\pi_{\rm r}$ is a strong quantum group morphism. 
The uniqueness property of $\pi_{\rm r}$ is evident.
\end{proof}

If now $\pi$ and $\pi'$ are two composable morphisms in $\CC$, 
it is straightforward to deduce from the uniqueness property  that 
$(\pi \circ \pi')_{\rm r} = \pi_{\rm r} \circ \pi_{\rm r}'$.  
Hence, we may define $R$ as follows. 

\begin{defnsub} 
The functor $R: \CC \to \RR$ takes each object $\AD$ in~$\CC$ to $(A_{\rm r},\D_{\rm r})$ in~$\RR$, 
and each morphism $\pi$ in~$\CC$ to the morphism $\pi_{\rm r}$ in~$\RR$. 
\end{defnsub}

\begin{propsub}\label{redpro} 
The functor $R$ is a left adjoint to the inclusion functor $\inc_\RR: \RR \to \,  \CC$, 
and the unit $\t$ of $R \dashv \inc_\RR$ is given by $\t_x=\t_A$ for each $x=\AD$ in $\CC$. 
In particular, $\RR$ is reflective in $\CC$.
\end{propsub}

\begin{proof} 
For each compact quantum group $x=\AD$, 
let $\t_x: x \to Rx$ be the morphism in $\CC$  given by $\t_x = \t_A$. 
Then Lemma \ref{redulem} implies that the map $\t$ which sends each $x$ to $\t_x$ 
is a natural transformation from $1_\CC$ to  $\inc_\RR \circ R$. 

Moreover, Lemma \ref{redulem} also implies that 
$(Rx,\t_x)$ is a universal morphism from $x$ to $\RR $ for each object $x = \AD$ in $\CC$.
Indeed,  consider an object  $y=(B,\D')$ in $\RR$ and a morphism $\pi: x \to y$ in $\CC$. 
Then $(B_{\rm r}, \D'_{\rm r}) = (B,\D')=y$ and $\t_B= \id_B$. 
Hence $ \pi_{\rm r}: Rx \to y$ is the unique morphism in $\RR$ such that 
$\pi = \pi_{\rm r} \circ \t_x$.

This shows that $R \dashv \inc_\RR$ and $\t$ is the unit of this adjunction.
\end{proof}

\subsubsection*{Universal compact quantum groups}

Let $\AD$ be a compact quantum group with associated Hopf $*$-algebra $\AA$.  
We  recall the construction of the universal compact quantum group 
associated to  $\AD$ (\cf~\cite[Section 3]{BMT} for more details). 

When $a \in \AA$, set $\nuu a = \sup_\phi \|\phi(a)\|$,
where the variable $\phi$ runs over all unital
homomorphisms $\phi$ from $\AA$ into any unital $C^*$-algebra $B$.
The function $\|\cdot\|_u :\AA\rightarrow [0,\infty ]$ is then a
$C^*$-norm on $\AA$ which majorises any other $C^*$-norm on~$\AA$.  
Let $\au$  be the $C^*$-algebra completion%
\footnote{As such a completion is unique only up to isomorphism, 
we actually make a choice here for each compact quantum group.} 
of $\AA$ with respect to the $C^*$-norm $\nuu{\cdot}$. 
As usual, we identify $\AA$ with its canonical copy inside $\au$. 
The $C^*$-algebra $\au$ has the universal property that every
unital homomorphism from $\AA$ to a unital $C^*$-algebra~$B$,
extends uniquely to a unital homomorphism from $\au$ to~$B$.

In particular, ${\D  : \AA\to \AA\odot \AA\subseteq \au\otimes \au}$ 
extends to a homomorphism 
\[
{\D_{\rm u} : \au\to \au\otimes \au},
\]
and $\ADu$ is then seen to be a compact quantum group, 
called the \emph{universal quantum group} of $\AD$.
Since $\AA$ is, by construction, a dense Hopf $*$-subalgebra of 
$\ADu$, it is the Hopf $*$-algebra associated to $\ADu$, by uniqueness.

By the universal property of $\au$, 
there is a canonical homomorphism  $\psi_A$  from $\au$ onto $A$ 
extending the identity map from $\AA$ to itself. 
Then $\D\circ \psi_A=(\psi_A\otimes\psi_A)\circ \D_{\rm u}$, 
and $h_A \circ \psi_A$ is the Haar state of $\ADu$, 
which just means that $\psi_A$ is a morphism in $\CC$ from $\ADu$ to $\AD$. 

A compact quantum group $\AD$ is called  \emph{universal} 
if  $\psi_A$ is injective. 
Equivalently, $\AD$ is universal if, and only if, 
the given  norm on $\AA$ is its greatest $C^*$-norm. 
Obviously, the  universal compact quantum group associated
to any $\AD$ is universal.

\begin{defnsub} 
The category $\UU$ is the full subcategory of $\CC$ 
whose objects are universal compact quantum groups.
\end{defnsub}

To see that universalization gives a functor $U$ from $\CC$ to $\UU$, we will use the following lemma.%
\footnote{As will be apparent from its proof, Lemma \ref{unilem} is also valid if we consider quantum group morphisms instead of strong quantum group morphisms.} 

\begin{lemsub} \label{unilem} 
Let $\pi$ be a strong quantum group morphism from  $\AD$ to $(B,\D')$. 
Then there exists a unique strong quantum group morphism 
$\pi_{\rm u}$ from  $\ADu$ to $(B_{\rm u} ,\D'_{\rm u})$ 
such that $\psi_B \circ \pi_{\rm u} = \pi\circ \psi_A$, 
that is, making the following diagram commute
\[
\xymatrix{
\ADu \ar[r]^{\psi_A} \ar[d]_{\pi_{\rm u}}
&\AD \ar[d]^\pi
\\
(B_{\rm u} ,\D'_{\rm u})\ar[r]_{\psi_B}
&(B,\D')
}
\]
%commute.
\end{lemsub}

\begin{proof}
We have $\pi : \AA \to \BB \subseteq B_{\rm u}$. 
Hence, by the universal property of $\au$, 
we may uniquely extend this map to a unital homomorphism $\pi_{\rm u} : \au \to B_{\rm u}$. 

Let  $a \in \AA$. 
Then  we have 
\[
(\pi_{\rm u} \otimes \pi_{\rm u}) (\D_{\rm u}(a)) 
= (\pi \otimes \pi) (\D(a))  
= \D' (\pi(a)) = \D'_{\rm u} (\pi_{\rm u} (a)).
\]
By density of $\AA$ and continuity, we see that $\pi_{\rm u}$ is a quantum group morphism. 
As $\pi_{\rm u}$ agrees with $\pi$ on $\AA$, $\pi_{\rm u}$ is injective on $\AA$. 
Hence,  $\pi_{\rm u}$ is a strong quantum group morphism. 

Further, as $\psi_B \circ \pi_{\rm u} =  \pi = \pi\circ \psi_A$ clearly holds on $\AA$, we have $\psi_B \circ \pi_{\rm u} = \pi\circ \psi_A$ (again by density of $\AA$ and continuity). 
Finally,  if $\phi$ is a another morphism which satisfies $\psi_B \circ \phi= \pi\circ \psi_A$, then $\phi$ agrees with $\pi$ on $\AA$, so $\phi =\pi_{\rm u}$. 
\end{proof}

If now $\pi$ and $\pi'$ are two composable morphisms in $\CC$, it is straightforward to deduce from the uniqueness property  that $(\pi \circ \pi')_{\rm u} = \pi_{\rm u} \circ \pi_{\rm u}'$.  
Hence, we may define $U$ as follows. 

\begin{defnsub} The functor $U : \CC \to \UU$ takes each object $\AD$ in~$\CC$ to $\ADu$ in~$\UU$, 
and each morphism $\pi$ in~$\CC$ to the morphism $\pi_{\rm u}$ in~$\UU$. 
\end{defnsub}

\begin{propsub}\label{unipro} The functor $U$ is a right adjoint to the inclusion functor $\inc_\UU : \UU \to \CC$ and the counit $\psi$ of the adjunction $\inc_\UU \dashv U$ is given by $\psi_y = \psi_B$ for each $y=(B, \D')$ in $\CC$. 
In particular, $\UU$ is coreflective in $\CC$.
\end{propsub}

\begin{proof} For each compact quantum group $y=(B, \D')$, let $\psi_y $ be the morphism in $\CC$ defined by $\psi_y = \psi_B$. Then Lemma \ref{unilem} implies that the map $\psi$ which sends each $y$ to $\psi_y$ is a natural transformation from $\inc_\UU \circ\, U$ to $1_\CC$. 

Lemma \ref{unilem} also implies that each $\psi_y$ is a universal morphism from $\UU $ to $y$ for each object $y=(B, \D')$ in $\CC$.
Indeed,  consider a universal compact quantum group $x=\AD$ and a morphism $\pi : x \to y$ in~$\CC$. Then $\psi_A$ is an isomorphism and $\pi' = \pi_{\rm u} \circ \psi_A^{-1} : \AD \to (B_{\rm u}, \D_{\rm u}')$ is clearly the unique morphism in $\UU(x, Uy)$ such that $\pi = \psi_y \circ \pi'$.
\end{proof}

\subsubsection*{Equivalence of $\RR$ and $\UU$}

It  follows from Propositions \ref{redpro} and \ref{unipro}
that $R|_\UU \dashv U|_\RR$ is an adjunction from $\UU$ to $\RR$. 
To see that this is an adjoint equivalence, we will use the following:

\begin{propsub} \label{fundpro} Let $\AD$ be a compact quantum group. Then: 
\begin{enumerate}
\item  $U\t_A= (\t_A)_{\rm u} $ is an isomorphism in $\UU$;
\item $R\psi_A = (\psi_A)_{\rm r} $ is an isomorphism in $\RR$.
\end{enumerate}
\end{propsub}

\begin{proof} (i)~Lemma \ref{unilem}, applied to $\t_A$, gives that  $$(\t_A)_{\rm u} : \ADu \to \big((A_{\rm r})_{\rm u}, (\D_{\rm r})_{\rm u}\big)$$ is a morphism in $\UU$ satisfying $\psi_{A_{\rm r}} \circ (\t_A)_{\rm u}  = \t_A \circ \psi_A$. 

Since $\t_A$ is injective on $\AA$, the map $\t_A(a) \mapsto a \in \AA \, \subseteq \au$ gives a 
well-defined homomorphism from $\t_A(\AA)$ to $\au$. Hence, by universality, it extends to
a homomorphism from $(A_{\rm r})_{\rm u} $ to $\au$, which is easily seen to be a morphism in $\UU$ and the inverse of $(\t_A)_{\rm u}$.
 
(ii)~Lemma \ref{redulem}, applied to $\psi_A$, gives that  $$(\psi_A)_{\rm r} : \big((\au)_{\rm r} , (\D_{\rm u})_{\rm r}\big) \to (A_{\rm r},\D_{\rm r})$$ is a morphism in $\RR$ satisfying $(\psi_{A})_{\rm r} \circ \t_{A_{\rm u}}  = \t_A \circ \psi_A$. 
As $(\psi_{A})_{\rm r} \circ \t_{A_{\rm u}}  = \t_A \circ \psi_A$ is surjective (because $\t_A $ and $ \psi_A$ are both surjective by construction), it is clear that $(\psi_{A})_{\rm r}$  is surjective. 
 
Moreover, as $h_{\rm r} \circ \t_A \circ \psi_A$ is the Haar state of $\ADu$ and $h_r$ is faithful, we have $\ker(\t_A \circ \psi_A) = N_{\au} = \ker(\t_{\au})$. Since $(\psi_{A})_{\rm r} \circ \t_{A_{\rm u}}  = \t_A \circ \psi_A$, it readily follows that $(\psi_{A})_{\rm r} $ is injective. 
 
Hence, $(\psi_{A})_{\rm r} $ is a bijection. But any quantum group morphism which is a bijection is easily seen to be an isomorphism in $\CC$. So $(\psi_{A})_{\rm r} $ is an isomorphism in~$\RR$.
\end{proof}

\begin{thmsub} \label{fundthmcqg}
The adjunction $R|_\UU \dashv U|_\RR$ is an adjoint equivalence.
In particular, the categories $\RR$ and $\UU$ are equivalent. 
\end{thmsub}

\begin{proof} 
Proposition \ref{fundpro} (i) (respectively
%~\ref{fundpro} 
(ii)) implies that the adjunction $R|_\UU \dashv U|_\RR$ satisfies condition (iv) (respectively~(viii)) in Theorem \ref{main}. As  $\UU$ (respectively $\RR$) is full, Theorem \ref{main} gives that the unit (respectively the counit) of this adjunction is a natural isomorphism. Hence, the assertion follows.
\end{proof}

\begin{remsub} Proposition \ref{fundpro} may be reformulated by saying that 
$(Rx, \theta_x\circ\psi_x)$ is an initial object  in $Ux\dn \RR$ and $(Ux, \theta_x\circ\psi_x)$ is a final object in $\UU \dn Rx$ for each $x\in\obj\CC$ (\cf~
%Remarks \ref{initialrem} and \ref{analoghyp}, and 
Theorem \ref{hyp-thm}), which means that Hypothesis~\ref{factor initial} and its dual analog are satisfied. Being surjective by construction,  $\psi_x$ is an epimorphism in $\CC$ for each $x \in \obj\CC$, so we can conclude from Corollary~\ref{faithful} that $U$ and $R$ are faithful.
\end{remsub}

\subsection{Other examples}
\label{other}

Here we describe two other examples of the maximal-normal equivalence, in which the subcategories are not only equivalent but in fact isomorphic. 
These concern tensor products and group representations, and it should be clear that one can readily construct an abundance of such examples.

\subsubsection*{Tensor products}
\label{tensor}

We show that the categories of maximal and minimal $C^*$-tensor products are equivalent, indeed isomorphic.%
\footnote{Of course, the tensor-product $C^*$-algebras themselves will usually not be isomorphic!}
More precisely, we show that, for a fixed $C^*$-algebra $D$, the categories of maximal tensor products $A\otimes_{\max} D$ and minimal tensor products $A\otimes_{\min} D$ are isomorphic.
We thank Chris Phillips for this suggestion.

We could easily have done everything with both variables free, \ie, allowing $D$ to vary as well as $A$, but we merely wanted to present examples, and the result we establish is more readily compared with the maximal-normal equivalence for coactions.
To see the relation, let $G$ be a locally compact group, and take $D=C^*(G)$. For any $C^*$-algebra $A$, let $\iota$ be the trivial action of $G$. Then the full and reduced crossed products are
\[
A\rtimes_\iota G=A\otimes_{\max} C^*(G)
\midtext{and}
A\rtimes_{\iota,r} G=A\otimes_{\min} C^*_r(G),
\]
and in each case the dual coaction is trivial.
The maximal-normal equivalence relates the maximal coaction $(A\rtimes_\iota G,\what\iota)$ to its normalization $(A\rtimes_{\iota,r} G,\what\iota^n)$, \ie, the maximal tensor product $A\otimes_{\max} C^*(G)$ to the minimal one $A\otimes_{\min} C^*_r(G)$, both with the trivial coaction.
The ``maximal-normal isomorphism'' we exhibit here relates only the $C^*$-algebras
$A\otimes_{\max} C^*(G)$ and $A\otimes_{\min} C^*(G)$ (not $A\otimes_{\min} C^*_r(G)$); thus the comparison is not perfect (and so even in with $D=C^*(G)$ the results we present here are not a special case of the maximal-normal equivalence for coactions), but clearly there is a strong similarity between the two types of equivalence.

Fix a $C^*$-algebra $D$. Our ambient category $\CC$
will comprise $C^*$-tensor products with $D$. More precisely, the objects in 
$\CC$ are pairs $(A,\sigma)$, where
$A$ is a $C^*$-algebra and $\sigma$ is a $C^*$-norm on the algebraic tensor product $A\odot D$;
and a morphism $\pi:(A,\sigma)\to (B,\tau)$ in $\CC$ is a $C^*$-homomorphism $\pi:A\to B$ such that
the homomorphism
\[
\pi\odot\id:A\odot D\to B\odot D
\]
between the algebraic tensor products is $\sigma-\tau$ bounded.

Thus, 
for any object $(A,\sigma)$ in $\CC$,
%$\id:(A,\max)\to (A,\sigma)$ and $\id:(A,\sigma)\to (A,\min)$ 
%%
$\id_A:(A,\max)\to (A,\sigma)$ and $\id_A:(A,\sigma)\to (A,\min)$ 
are morphisms in $\CC$, where $\max$ and $\min$ denote the maximal and minimal $C^*$-norms, respectively.
Also, any $C^*$-homomorphism $\pi:A\to B$ gives two morphisms $\pi:(A,\max)\to (B,\max)$ and $\pi:(A,\min)\to (B,\min)$ in $\CC$.

A moment's thought reveals that $\CC$ really is a category: the identity morphism on an object $(A,\sigma)$ is $\id_A$, and the composition of morphisms $\pi:(A,\sigma)\to (B,\tau)$ and $\phi:(B,\tau)\to (C,\gamma)$ is $\phi\circ\pi:(A,\sigma)\to (C,\gamma)$.

Our subcategories $\MM$ and $\NN$ will comprise the maximal and minimal tensor products, respectively.
That is, $\MM$ is the full subcategory of~$\CC$
with objects of the form $(A,\max)$, 
and $\NN$ is the full subcategory 
with objects of the form $(A,\min)$.
The following proposition is almost trivial.

\begin{propsub}
The subcategories $\MM$ and $\NN$ of $\CC$ are coreflective and reflective, respectively.
\end{propsub}

\begin{proof}
To show that $\MM$ is coreflective, we must construct a right adjoint $M$ of the inclusion functor $\inc_\MM:\MM\to\CC$.
It suffices to find, for each object $(A,\sigma)$ in $\CC$, a universal morphism 
$(M(A,\sigma),\psi_{(A,\sigma)})$
from $\MM$ to $(A,\sigma)$, because there would then be a unique way to extend $M$ to a right adjoint such that $\psi:\inc_\MM\circ M:\to 1_\CC$ 
is a natural transformation.
So, let $(B,\max)$ be an object in $\MM$, and let
$
\pi:(B,\max)\to (A,\sigma)
$
be a morphism. 

Then obviously $\pi:(B,\max)\to (A,\max)$ is the unique morphism making the diagram
\[
\xymatrix{
(B,\max) \ar[dr]^\pi \ar@{-->}[d]_\pi^{!}
\\
(A,\max) \ar[r]_{\id_A}
&(A,\sigma)
}
\]
commute, so we can take
\[
M(A,\sigma)=(A,\max)\midtext{and}
%\epsilon_{(A,\sigma)}=\id_A
%%
\psi_{(A,\sigma)}=\id_A
\]
It is just as easy to construct a left adjoint $N$ for the inclusion functor $\inc_\NN:\NN\to\CC$.
\end{proof}

Note that the adjunctions $\inc_\MM\dashv M$ and $N\dashv \inc_\NN$ implicitly constructed in the above proof are given by
\begin{align*}
M(A,\sigma)&=(A,\max)&M\pi&=\pi
\\
N(A,\sigma)&=(A,\min)&N\pi&=\pi,
\end{align*}
where $(A,\sigma)$ and $\pi$ are an object and a morphism, respectively, of $\CC$.
Moreover, the counit of $\inc_\MM\dashv M$ and the unit of $N\dashv \inc_\NN$ are both given by identity maps:
\begin{align*}
%\epsilon_{(A,\sigma)}&=\id_A:(A,\max)\to (A,\sigma)
%%
\psi_{(A,\sigma)}&=\id_A:(A,\max)\to (A,\sigma)
\\
%\eta_{(A,\sigma)}&=\id_A:(A,\sigma)\to (A,\min)
%%
\theta_{(A,\sigma)}&=\id_A:(A,\sigma)\to (A,\min)
\end{align*}
We could now apply the results of Sections~\ref{equivalence} and \ref{boilerplate},
after verifying 
the relevant hypotheses
therein, but in this context all this reduces to almost a triviality, and in fact the restriction $N|_\MM$ is not only an equivalence, but in fact an isomorphism of subcategories:
% 
% \begin{propsub}
% With the above notation, the
% restriction $N|_\MM:\MM\to\NN$
% is an isomorphism of categories.
% \end{propsub}
% 
% \begin{proof}
% This is obvious; 
$N|_\MM$ and $M|_\NN$ are easily seen to be inverses of each other.
%\end{proof}

\begin{remsub}
The components of both the counit 
$\psi:\inc_\MM\circ M\to 1_\CC$ and $\theta:1_\CC\to \inc_\NN\circ N$ 
are both monomorphisms and epimorphisms, since these components reduce to $\id_A$ for each object $(A,\sigma)$.
Of course, in spite of all this we must still keep in mind that neither the counit nor the unit is an isomorphism.
\end{remsub}

\subsubsection*{Group representations}
\label{groups}

Another example of the ``maximal-normal isomorphism'' is given by group representations weakly containing the trivial representation.
More precisely, this time our ambient category $\CC$ will have
\begin{itemize}
\item objects: triples $(G,u,A)$, where $G$ is a locally compact group, $A$ is a $C^*$-algebra, and $u:G\to M(A)$ is a strictly continuous unitary homomorphism that weakly contains the trivial representation $1_G:G\to\C$ (given by $1_G(s)=1$ for all $s\in G$), and for which the associated morphism $\pi_u:C^*(G)\to A$ in \cs\ maps $C^*(G)$ onto $A$;

\item morphisms: $(\phi,\pi):(G,u,A)\to (H,v,B)$ in $\CC$ means that $\phi:G\to H$ is a continuous homomorphism, $\pi:A\to B$ is a morphism in \cs\, and the diagram
\[
\xymatrix{
G \ar[r]^-u \ar[d]_\phi
&M(A) \ar[d]_{\bar\pi}
&A \ar@{_(->}[l] \ar[dl]^\pi
\\
H \ar[r]_-v
&M(B)
}
\]
commutes.
\end{itemize}

Thus, for each object $(G,u,A)$ in $\CC$, the weak containment hypothesis means that there is a morphism $\gamma_u$ in \cs\ making the diagram
\[
\xymatrix{
G \ar[r]^-u \ar[dr]_{1_G}
&M(A) \ar[d]_{\bar{\gamma_u}}
&A \ar[dl]^{\gamma_u} \ar@{_(->}[l]
\\
&\C
}
\]
commute.
It is routine to check that this is a category.

This time, our full subcategories $\MM$ and $\NN$ will have objects of the form $(G,i_G,C^*(G))$ and $(G,1_G,\C)$, respectively, where $i_G:G\to M(C^*(G))$ is the canonical inclusion.

\begin{propsub}
The subcategories $\MM$ and $\NN$ of $\CC$ are coreflective and reflective, respectively.
\end{propsub}

\begin{proof}
As usual, it suffices to find, for each object $(G,u,A)$ of $\CC$, a universal morphism $(M(G,u,A),\psi_{(G,u,A)})$ from $\MM$ to $(G,u,A)$, and a universal morphism $(N(G,u,A),\theta_{(G,u,A)})$ from $(G,u,A)$ to $\NN$. Note that we have a commutative diagram
\[
\xymatrix{
&M(C^*(G))
\ar[d]_{\bar{\pi_u}}
&C^*(G)
\ar[dl]^{\pi_u}
\ar@{_(->}[l]
\\
G \ar[ur]^{i_G} \ar[r]^-u \ar[dr]_{1_G}
&M(A) \ar[d]_{\bar{\gamma_u}}
&A
\ar[dl]^{\gamma_u}
\ar@{_(->}[l]
\\
&\C.
}
\]
\textbf{Claim:} it follows that we can take
\begin{align*}
M(G,u,A)&=(G,i_G,C^*(G))\\
\psi_{(G,u,A)}&=(\id_G,\pi_u)\\
N(G,u,A)&=(G,1_G,\C)\\
\theta_{(G,u,A)}&=(\id_G,\gamma_u).
\end{align*}
To verify the claim, first we show that
\[
(\id_G,\pi_u):(G,i_G,C^*(G))\to (G,u,A)
\]
is final in the comma category $\MM\dn (G,u,A)$:

Given an object $(H,i_H,C^*(H))$ in $\MM$ and a morphism
$(\phi,\omega):(H,i_H,C^*(H))\to (G,u,A)$, we must show that the diagram 
\[
\xymatrix{
(H,i_H,C^*(H)) \ar[dr]^{(\phi,\omega)} \ar@{-->}[d]_{(\sigma,\tau)}^{!}
\\
(G,i_G,C^*(G)) \ar[r]_-{(\id_G,\pi_u)}
&(G,u,A)
}
\]
can be uniquely completed.
We will show that we can take
\[
(\sigma,\tau)=(\phi,C^*(\phi)),
\]
where $C^*(\phi):C^*(H)\to C^*(G)$ is the morphism in \cs\ corresponding to the continuous homomorphism $\phi:H\to G$.
First of all, note that $(\phi,C^*(\phi)):(H,i_H,C^*(H))\to (G,i_G,C^*(G))$ is a morphism in $\CC$ (in fact, in $\MM$, since $\MM$ is full and both objects are in $\MM$), by the universal property of group $C^*$-algebras.
Of course
$
\phi=\id_G\circ\phi,
$
so it remains to show that
\[
\omega=\pi_u\circ C^*(\phi)
\]
in \cs.
This time, because we are ``mixing categories'', we take some care with the ``barring'' of nondegenerate homomorphisms into multiplier algebras (see \cite[Appendix~A]{boiler}).
So, we must show that
\[
\bar{\pi_u}\circ C^*(\phi)=\omega:C^*(H)\to M(A).
\]
Since all the above homomorphisms are nondegenerate, it suffices to show that
\[
\bar{\bar{\pi_u}\circ C^*(\phi)}=\bar\omega:M(C^*(H))\to M(A).
\]
Furthermore, by the universal property of group $C^*$-algebras it suffices to show that the above equation holds after pre-composing both sides with $i_H:H\to M(C^*(H))$:
\begin{align*}
\bar{\bar{\pi_u}\circ C^*(\phi)}\circ i_H
&=\bar{\pi_u}\circ \bar{C^*(\phi)}\circ i_H
&&\text{(properties of barring)}
\\&=\bar{\pi_u}\circ i_G\circ \phi
&&\text{($(\phi,C^*(\phi))$ is a morphism)}
\\&=u\circ \phi
&&\text{(universal property of $\pi_u$)}
\\&=\bar\omega\circ i_H
&&\text{($(\phi,\omega)$ is a morphism)}.
\end{align*}

To finish, we need to verify that
\[
(\id_G,\gamma_u):(G,u,A)\to (G,1_G,\C)
\]
is initial in the comma category $(G,u,A)\dn \NN$:
given an object $(H,1_H,\C)$ in $\NN$ and a morphism
$(\phi,\omega):(G,u,A)\to (H,1_H,\C)$,
we must show that the diagram
\[
\xymatrix{
(G,u,A) 
\ar[r]^-{(\id_G,\gamma_u)}
\ar[dr]_{(\phi,\omega)}
&(G,1_G,\C) 
\ar@{-->}[d]^{(\sigma,\tau)}_{!}
\\
&(H,1_H,\C)
}
\]
can be uniquely completed.
We will show that we can take 
$(\sigma,\tau)=(\phi,\id_\C)$.
Again, the only nontrivial thing to show is
$\omega=\id_\C\circ \gamma_u=\gamma_u$. 
Note that $\gamma_u:A\to\C$ is the unique homomorphism such that
$\bar{\gamma_u}\circ u=1_G$. 
Thus the following computation finishes the proof:
\[
\bar\omega\circ u
=1_H\circ \phi
=1_G.
\qedhere
\]
\end{proof}

\begin{propsub}
With the above notation, the
restriction $N|_\MM:\MM\to\NN$
is an isomorphism of categories.
\end{propsub}

\begin{proof}
Again 
$N|_\MM$ and $M|_\NN$ are easily seen to be inverses of each other. 
\end{proof}

\appendix

\section{}

In this appendix we take the opportunity to reinterpret much of the existing theory of coaction crossed products
in the present, more categorical, context.

To begin, we explicitly record a few properties of the category \cs. 
Since \coact\ is obtained from \cs\ by adding extra structure, some of the following observations will be relevant for \coact\ as well.

A morphism $\phi:A\to B$ in~\cs\ is a monomorphism if and only it is injective.
Thus, monomorphicity is completely determined by the kernel. What about epimorphicity? One direction is elementary:
If $\phi:A\to B$ in \cs\ and $\phi$ is surjective (\ie, $\phi(A)=B$), then $\phi$ is an epimorphism.
Of course, the converse is false for general morphisms in \cs. 
For example, if $\phi(A)$ properly contains $A$ then $\phi$ is an epimorphism in~\cs. There is a positive result, which does not seem to have become a standard tool among operator algebraists:

\begin{lem}
[{\cite{epi}}]
Suppose $\phi:A\to B$ is a homomorphism --- so we are requiring $\phi$ to map $A$ into $B$ itself rather than merely $M(B)$.   
Then $\phi$ is an epimorphism in~\cs\ if and only if it is surjective.
\end{lem}

The above results lead to an obvious question, which does not seem to be addressed in the literature:
if $\phi:A\to B$ is an epimorphism in~\cs, must $\phi(A)\supset B$?

Factorizability in~\cs\ is also often controlled by kernels:
if $\phi:A\to B$ and $\psi:A\to C$ in~\cs, with $\phi$ surjective, then there is a morphism $\rho$ in~\cs\ making the diagram
\[
\xymatrix{
A \ar[r]^-\phi \ar[dr]_\psi
&B \ar@{-->}[d]^\rho
\\
&C
}
\]
commute if and only if $\ker\phi\subset\ker\psi$, and moreover $\rho$ is unique.
We do not know whether the conclusion still holds if we weaken the surjectivity hypothesis on $\phi$ to epimorphicity of $\phi$ in~\cs.

The above factorizability criterion carries over to \coact\ (a routine diagram chase shows the required equivariance):
if $\phi:(A,\delta)\to (B,\epsilon)$ and $\psi:(A,\delta)\to (C,\gamma)$ in~\coact, with $\phi$ surjective,
then there is a morphism $\rho$ in~\coact\ making the diagram
\[
\xymatrix{
(A,\delta) \ar[r]^-\phi \ar[dr]_\psi
&(B,\epsilon) \ar@{-->}[d]^\rho
\\
&(C,\gamma)
}
\]
if and only if $\ker\phi\subset\ker\psi$,
and moreover $\rho$ is unique.

\begin{lem}
[{\cite[Lemma~1.11]{fullred}}]
Every morphism $\mu:C_0(G)\to B$ in~\cs\ implements an inner coaction $\ad\mu$ of $G$ on $B$, and all inner coactions are normal.
\end{lem}

Of course ``$\ad\mu$'' is an abuse of notation --- it is intended to be in an obvious way dual to the notation $\ad u$ for the inner action determined by a strictly continuous unitary homomorphism $u:G\to M(B)$. The notation stands for the morphism $\ad\mu:B\to B\otimes C^*(G)$ in~\cs\ defined by
\[
\ad\mu(b)=\ad\overline{\mu\otimes\id}(w_G)(b\otimes 1),
\]
where $w_G$ denotes the unitary element of $M(C_0(G)\otimes C^*(G))$ determined by the canonical embedding $G\hookrightarrow M(C^*(G))$.

\begin{lem}
\label{covariant}
Let $(A,\delta)$ be a coaction, and let $\mu:C_0(G)\to B$ and $\pi:A\to B$ in~\cs. Then the pair $(\pi,\mu)$ is a covariant homomorphism of $(A,\delta)$ in $M(B)$ if and only if $\pi$ is $\delta-\ad\mu$ equivariant.
\end{lem}

We will show presently that the crossed-product functor for coactions is right-adjointable.
Of course, this will follow from the universal property of crossed products.
It will be a little clearer to consider this universality for an individual coaction first.
Let's recall Raeburn's definition of crossed product:

\begin{defn}
\label{crossed product}
Let $(\eta,\nu)$ be a covariant homomorphism of a coaction $(A,\delta)$ in $M(C)$. Then $(C,\eta,\nu)$ is a \emph{crossed product} of $(A,\delta)$ if for every   covariant homomorphism $(\pi,\mu)$ of $(A,\delta)$ in $M(B)$ there is a unique morphism $\rho$ in~\cs\ making the diagram
\begin{equation}
\label{covariance}
\xymatrix{
A \ar[r]^-\eta \ar[dr]_\pi
&C \ar@{-->}[d]^\rho_{!}
&C_0(G) \ar[l]_-\nu \ar[dl]^\mu
\\
&B
}
\end{equation}
commute.
The existence of $\rho$ is expressed by saying that $(\pi,\mu)$ \emph{factors through $(\eta,\nu)$}, and 
existence and uniqueness together are expressed by saying that $(\pi,\mu)$ \emph{factors uniquely through  $(\eta,\nu)$}.
\end{defn}

\begin{rem}
In addition to the above axioms, Raeburn explicitly hypothesizes that the $C^*$-algebra $C$ is generated by products of the form $\eta(a)\nu(f)$ for $a\in A$ and $f\in C_0(G)$. This hypothesis is redundant: the theory of crossed products tells us that if $(C,\eta,\nu)$ and $(D,\sigma,\omega)$ are crossed products of $(A,\delta)$, then there is a unique isomorphism $\theta:C\to D$ such that $\theta\circ\eta=\sigma$ and $\theta\circ\nu=\omega$ in~\cs. Since there is at least \emph{one} crossed product $(C,\eta,\nu)$ for which $C$ is generated by\footnote{in fact is the \emph{closed span of}} the set of products $\eta(A)\nu(C_0(G))$,
it must therefore be true for \emph{every} crossed product $(D,\sigma,\omega)$.
That being said, we can nevertheless turn this redundancy around to find a useful replacement for the uniqueness clause:
\end{rem}

\begin{lem}
\label{generated}
Let $(\eta,\nu)$ be a covariant homomorphism of a coaction $(A,\delta)$ in $M(C)$, and suppose that every covariant homomorphism of $(A,\delta)$ factors through $(\eta,\mu)$. Then $(C,\eta,\nu)$ is a crossed product of $(A,\delta)$ if and only if $C$ is generated by $\eta(A)\nu(C_0(G))$.
\end{lem}

We will use \lemref{covariant} 
to show that crossed products give universal morphisms.
We need a functor:

\begin{notn}
$\ad$ denotes the functor that takes an object $(B,\mu)$ of $C_0(G)\dn \cs$ to the object $(B,\ad\mu)$ of~\coact, and takes a morphism $\psi$ in $C_0(G)\dn \cs$ to $\psi$, now regarded as a morphism in~\coact.
\end{notn}

Note that the above definition of $\ad$ makes sense on morphisms, because if $\psi:(B,\mu)\to (C,\nu)$ in $C_0(G)\dn \cs$ then
(computing in the usual category of $C^*$-algebras and $*$-homomorphisms)
\begin{align*}
 \ad\bar{\nu\otimes\id}(w_G)\bigl(\psi(b)\otimes 1\bigr)
&= \bar{\psi\circ\mu\otimes\id}(w_G)\bigl(\psi(b)\otimes 1\bigr)\\
&= \bar{\psi\otimes\id}\bigl(\ad\bar{\mu\otimes\id}(b\otimes 1)\bigr),
\end{align*}
so that $(\ad\nu)\circ\psi = (\psi\otimes\id)\circ(\ad\mu)$ in~\cs. 

\begin{lem}
\label{crossed product universal}
Let $(A,\delta)$ be a coaction of $G$, 
let $(C,\nu)$ be an object in $C_0(G)\dn \cs$, 
and let $\eta:(A,\delta)\to (C,\ad\nu)$ in~\coact.
If $(C,\eta,\nu)$ is a crossed product of $(A,\delta)$ 
then $(C,\nu,\eta)$ is a universal morphism from $(A,\delta)$ to the functor $\ad$.
\end{lem}

\begin{proof}
Let $(B,\mu)$ be an object in $C_0(G)\dn \cs$, and let $\pi:(A,\delta)\to (B,\ad\mu)$ in~\coact.
By \lemref{covariant} the pair $(\pi,\mu)$ is a covariant homomorphism of $(A,\delta)$ in $M(B)$. Thus, since $(C,\eta,\nu)$ is a crossed product of $(A,\delta)$ there is a unique morphism $\rho:C\to B$ in~\cs\ making diagram~\eqref{covariance} commute. Then $\rho:(C,\nu)\to (B,\mu)$ in $C_0(G)\dn \cs$, so $\rho$ is also a morphism from $(C,\ad\nu)$ to $(B,\ad\mu)$ in~\coact, and it is the unique such morphism making the diagram
\begin{equation}
\label{universal ad}
\xymatrix{
(A,\delta) \ar[r]^-\eta \ar[dr]_\pi
&(C,\ad\nu) \ar@{-->}[d]^\rho_{!}
\\
&(B,\ad\mu)
}
\end{equation}
commute. We have shown that $(C,\nu,\eta)$ is a universal morphism from $(A,\delta)$ to $\ad$.
\end{proof}

\begin{q}
Is the converse of the above lemma true? That is, if $(C,\nu,\eta)$ is a universal morphism from $(A,\delta)$ to $\ad$, is $(C,\eta,\nu)$ a crossed product of $(A,\delta)$? The naive approach would be to take any covariant homomorphism $(\pi,\mu)$ of $(A,\delta)$ in $M(B)$, then note that by \lemref{covariant} we have a morphism $\pi:(A,\delta)\to (B,\ad\mu)$ in~\coact, so by universality we have a unique morphism $\rho$ in~\coact\ making the diagram~\eqref{universal ad} commute. But this only says that the coactions $\ad\mu$ and $(\ad\rho)\circ\nu$ on $B$ coincide, which does not imply that $\mu=\rho\circ\nu$.
\end{q}

Now we promote the above universal property to a functor: 
once we choose a universal morphism $(A\times_\delta G,j_G,j_A)$ for each coaction $(A,\delta)$, there is a unique functor from $\coact$ to $C_0(G)\dn\cs$ that takes an object $(A,\delta)$ to $(A\times_\delta G,j_G)$ and is a left adjoint to the functor $\ad$.

\begin{defn}
The above functor, taking $(A,\delta)$ to $(A\times_\delta G,j_G)$, is the
\emph{crossed-product functor}, denoted by
\[
\cp:\coact\to C_0(G)\dn \cs.
\]
The value of $\cp$ on a morphism $\phi$ in~\coact\ is written $\phi\times G$.
\end{defn}

The above discussion can be summarized by:

\begin{cor}
$\cp$ is a left adjoint for $\ad$, with unit
$j:1_{\coact}\to \ad\circ\cp$ given by
\[
j_{(A,\delta)}=j_A:(A,\delta)\to (A\times_\delta G,j_G).
\]
\end{cor}

The following easy observation is another summary of the above discussion, and follows from essential uniqueness of universal morphisms:

\begin{cor}
\label{recognize crossed product}
Let 
$(\eta,\nu)$ be a covariant homomorphism of a coaction $(A,\delta)$ in $M(C)$.
Then $(C,\nu,\eta)$ is a crossed product of $(A,\delta)$ if and only if
\[
\eta\times\nu:A\times_\delta G\to C
\]
is an isomorphism in~\cs.
\end{cor}

\begin{lem}
[``Epi-mono factorization'' in~\coact]
\label{epi-mono}
If $\phi:(A,\delta)\to (B,\epsilon)$ in~\coact, then there is a unique coaction $\gamma$ on $\phi(A)$ such that the diagram
\[
\xymatrix{
(A,\delta) \ar[r]^\phi \ar[dr]_\phi
&(\phi(A),\gamma) \ar@{^(->}[d]
\\
&(B,\epsilon)
}
\]
commutes in~\coact. Moreover, $\gamma$ is normal if $\epsilon$ is.
\end{lem}

\begin{proof}
We claim that $\gamma=\bar\epsilon|_{\phi(A)}$ does the job.
First, note that
\begin{align*}
\bar\epsilon(\phi(A))
&=\bar{\phi\otimes\id}(\delta(A))
\\&\subset \bar{\phi\otimes\id}(M(A\otimes G^*(G))
\\&\subset M(\phi(A)\otimes C^*(G)).
\end{align*}
Further, 
$\gamma$ is injective since $\bar\epsilon$ is injective on $M(B)$,
and
$\gamma$ satisfies the coaction identity because $\epsilon$ does. Finally,
\begin{align*}
\gamma\bigl(\phi(A)\bigr)\bigl(1_{M(\phi(A))}\otimes C^*(G)\bigr)
&=\bar\epsilon\bigl(\phi(A)\bigr)\bigl(1_{M(B)}\otimes C^*(G)\bigr)
\\&=\bar{\phi\otimes\id}\bigl(\delta(A)\bigr)\bigl(1_{M(B)}\otimes C^*(G)\bigr)
\\&=\bar{\phi\otimes\id}\Bigl(\delta(A)\bigl(1_{M(A)}\otimes C^*(G)\bigl)\Bigr),
\end{align*}
which has closed span $\phi(A)\otimes C^*(G)$ since $\delta(A)(1\otimes C^*(G))$ has closed span $A\otimes C^*(G)$.

For the last part, suppose that $\epsilon$ is normal. The inclusion map gives a morphism $\iota:(\phi(A),\gamma)\hookrightarrow (B,\epsilon)$, so we have a commutative diagram
\[
\xymatrix{
\phi(A) \ar[r]^\iota \ar[d]_{j_{\phi(A)}}
&B \ar[d]^{j_B}
\\
\phi(A)\times_\gamma G \ar[r]_{\iota\times G}
&B\times_\epsilon G
}
\]
in~\cs.
Since $j_B$ and $\iota$ are injective, so is $j_{\phi(A)}$, so $\gamma$ is normal.
\end{proof}

As a consequence of the above, we have another effective means to recognize normal coactions:

\begin{cor}
[{\cite[Lemma~2.2 and Proposition~2.3]{fullred}}]
\label{unitary}
If $(\pi,\mu)$ is a covariant homomorphism of $(A,\delta)$ in $M(B)$, then $\bar{\ad\mu}$ restricts to a normal coaction $\delta^\mu$ on $\pi(A)$, and then $\pi:(A,\delta)\to (\pi(A),\delta^\mu)$ in~\coact. Moreover, if $\pi$ is injective then it is an isomorphism of $(A,\delta)$ onto $(\pi(A),\delta^\mu)$. Thus, $(A,\delta)$ is normal if and only if it has a covariant homomorphism $(\pi,\mu)$ with $\pi$ injective.
\end{cor}

\begin{cor}
[{\cite[Proposition~2.5]{fullred}}]
\label{j_A crossed}
For every coaction $(A,\delta)$,
the morphism
\[
j_A\times G:A\times_\delta G\to j_A(A)\times_{\ad j_G} G
\]
is an isomorphism.
\end{cor}

\begin{proof}
Since $j_A:A\to j_A(A)$ is surjective, by \lemref{generated} and \corref{recognize crossed product} it suffices to show that every covariant homomorphism 
of $(A,\delta)$ factors through $(j_{j_A(A)}\circ j_A,j_G)$, and then by \lemref{covariant} it suffices to show that
for
every morphism $\pi:(A,\delta)\to (B,\ad\mu)$ to an inner coaction 
there is a morphism $\rho$ making the diagram
\begin{equation}
\label{j_A factor}
\xymatrix{
(A,\delta) \ar[r]^-{j_A} \ar[dr]_\pi
&(j_A(A),\ad j_G) \ar@{-->}[d]^\rho
\\
&(B,\ad\mu)
}
\end{equation}
commute in~\coact.
Define
\[
\rho=\bar{\pi\times\mu}|_{j_A(A)}:j_A(A)\to M(B).
\]
Applying \corref{unitary} to the morphism
\[
j_A:(A,\delta)\to (A\times_\delta G,\ad j_G),
\]
then post-composing with the morphism
\[
\pi\times\mu:(A\times_\delta G,\ad j_G)\to (B,\ad\mu),
\]
we see that 
$\rho$ 
gives a 
suitable
morphism in~\coact.
\end{proof}

%----------------------------

%\bibliographystyle{amsplain}
%\bibliography{subcat}

\providecommand{\bysame}{\leavevmode\hbox to3em{\hrulefill}\thinspace}
\providecommand{\MR}{\relax\ifhmode\unskip\space\fi MR }
\providecommand{\MRhref}[2]{%
  \href{http://www.ams.org/mathscinet-getitem?mr=#1}{#2}
}
\providecommand{\href}[2]{#2}

\end{document}